\documentclass[12pt]{amsart}

\usepackage{amsmath}
\usepackage{amssymb}
\usepackage{ascmac}
\usepackage{amsthm}

\makeatletter

\@addtoreset{equation}{section}
\makeatother 

\theoremstyle{plain}
\newtheorem{thm}{Theorem}[section]
\newtheorem*{thm*}{Theorem}
\newtheorem{prop}[thm]{Proposition}
\newtheorem{lem}[thm]{Lemma}
\newtheorem{cor}[thm]{Corollary}

\theoremstyle{definition}
\newtheorem{ex}{Example}[section]
\newtheorem{defi}[ex]{Definition}

\theoremstyle{remark}
\newtheorem{rem}{Remark}[section]

\newcommand{\vol}{\operatorname{vol}}
\newcommand{\ric}{\operatorname{Ric}}

\newcommand{\Hess}{\operatorname{Hess}}

\newcommand{\IR}{\operatorname{InRad}}
\newcommand{\PI}{\operatorname{PartInRad}}
\newcommand{\OI}{\operatorname{ObsInRad}}
\newcommand{\SE}{\operatorname{Sep}}
\newcommand{\DS}{\operatorname{BSep}}
\newcommand{\HS}{\operatorname{HS}}

\newcommand{\inte}{\mathrm{Int}\,}

\newcommand{\supp}{\mathrm{supp}\,}

\newcommand{\ball}{B^{n}_{\kappa,\lambda}}
\newcommand{\const}{C_{\kappa,\lambda}}

\newcommand{\bm}{\partial M}

\usepackage{latexsym}

\title[$1$-Lipschitz functions on manifolds with boundary]{Concentration of $1$-Lipschitz functions\\ on manifolds with boundary\\ with Dirichlet boundary condition}
\author{Yohei Sakurai}
\date{August 16, 2018}
\address{Institute for Applied Mathematics, University of Bonn, Endenicher Allee 60, D-53115 Bonn, Germany}
\email{sakurai@iam.uni-bonn.de}

\begin{document}
\maketitle

\begin{abstract}
In this paper,
we consider a concentration of measure problem on Riemannian manifolds with boundary.
We study concentration phenomena of non-negative $1$-Lipschitz functions with Dirichlet boundary condition around zero,
which is called boundary concentration phenomena.
We first examine relation between boundary concentration phenomena and large spectral gap phenomena of Dirichlet eigenvalues of Laplacian.
We will obtain analogue of the Gromov-V. D. Milman theorem and the Funano-Shioya theorem for closed manifolds.
Furthermore,
to capture boundary concentration phenomena,
we introduce a new invariant called the observable inscribed radius.
We will formulate comparison theorems for such invariant under a lower Ricci curvature bound,
and a lower mean curvature bound for the boundary.
Based on such comparison theorems,
we investigate various boundary concentration phenomena of sequences of manifolds with boundary.
\end{abstract}

\section{Introduction}
In the present paper,
we consider a concentration of measure problem on manifolds with boundary.
We study concentration phenomena of non-negative $1$-Lipschitz functions with Dirichlet boundary condition.

%%%%%%%%%%%%%%%%%%%%%%%%
\subsection{Motivations}
Let us recall the following well-known fact:
The normalized volume measure on the $n$-dimensional unit sphere concentrates around the equator when $n$ is large.
One can rephrase this fact as follows:
The normalized volume measure on the $n$-dimensional unit hemisphere concentrates around the boundary when $n$ is large.

We call a triple $X=(X,d_{X},\mu_{X})$ a \textit{(smooth) metric measure space with boundary}
when $X$ is a connected complete Riemannian manifold with boundary,
$d_{X}$ is the Riemannian distance,
and $\mu_{X}$ is a Borel probability measure on $X$.
Let $\partial X$ denote its boundary.
In this paper,
we consider the following problem:
\textit{For a given sequence $\{X_{n}\}$ of metric measure spaces with boundary $X_{n}=(X_{n},d_{X_{n}},\mu_{X_{n}})$,
does the measure $\mu_{X_{n}}$ concentrate around $\partial X_{n}$ when $n$ is large ?}
We will observe that
$\mu_{X_{n}}$ concentrates around $\partial X_{n}$ if and only if
every $1$-Lipschitz function $\varphi_{n}:X_{n}\to [0,\infty)$ with $\varphi_{n}|_{\partial X_{n}}=0$ is closed to zero (more precisely, see Remark \ref{rem:characterization of Dirichlet Levy} and Proposition \ref{prop:characterization of Dirichlet Levy}).
From this point of view,
we investigate concentration phenomena of non-negative $1$-Lipschitz functions with Dirichlet boundary condition around zero.
We call such phenomena \textit{boundary concentration phenomena}.

%%%%%%%%%%%%%%%%%%%%%%%%
\subsection{Observable inscribed radii}
Gromov \cite{G} has established theory of geometry of metric measure spaces
based on the idea of concentration of measure phenomena discovered by L\'evy \cite{Le}, and developed by V. D. Milman \cite{M1}, \cite{M2}.
He has introduced some important invariants on metric measure spaces.
One of them is the so-called \textit{observable diameter}
that measures the difference between $1$-Lipschitz functions and constants.
The observable diameter has been widely studied from the view point of the study of concentration phenomena of $1$-Lipschitz functions (see e.g., \cite{G}, \cite{L2}, \cite{Sh} and the references therein).

We now introduce a new invariant on metric measure spaces with boundary called the \textit{observable inscribed radius}
that measures the difference between non-negative $1$-Lipschitz functions with Dirichlet boundary condition and zero.
We will refer to the formulation of the observable diameter on metric measure spaces.

Let $X=(X,d_{X},\mu_{X})$ be a metric measure space with boundary.
Let $\rho_{\partial X}:X\to [0,\infty)$ stand for the distance function from the boundary $\partial X$ defined as $\rho_{\partial X}(x):=d_{X}(x,\partial X)$.
The function $\rho_{\partial X}$ is $1$-Lipschitz with $\rho_{\partial X}|_{\partial X}=0$.
The \textit{inscribed radius $\IR X$ of $X$} is defined to be the supremum of the distance function $\rho_{\partial X}$ over $X$.
We extend the notion of the inscribed radius to all subsets of $X$.
For $\Omega \subset X$,
we define the \textit{inscribed radius $\IR \Omega$ of $\Omega$} as follows:
If $\Omega \neq \emptyset$,
then
\begin{equation}\label{eq:inscribed radius}
\IR \Omega:=\sup_{x \in \Omega} \rho_{\partial X}(x);
\end{equation}
if $\Omega=\emptyset$,
then $\IR \Omega:=0$.
For $\xi \in (-\infty,1]$,
let us define the \textit{$\xi$-partial inscribed radius $\PI(X;\xi)$ of $X$} by
\begin{equation}\label{eq:partial inscribed radius}
\PI(X;\xi):=\inf_{\Omega \subset X}\,\IR \Omega,
\end{equation}
where the infimum is taken over all Borel subsets $\Omega$ with $\mu_{X}(\Omega) \geq \xi$.

We set a screen $I:= [0,\infty)$.
For a $1$-Lipschitz function $\varphi:X\to I$ with $\varphi|_{\partial X}=0$,
we call the metric measure space with boundary
\begin{equation}\label{eq:screen}
I_{\varphi}:=\left(I,d_{I},m_{I,\varphi}\right)
\end{equation}
the \textit{$\varphi$-screen},
where $m_{I,\varphi}$ denotes the push-forward $\varphi_{\#}\mu_{X}$ of $\mu_{X}$ by $\varphi$.

We now define the following quantity:
\begin{defi}\label{defi:observable inscribed radius}
For $\eta>0$,
we define the \textit{$\eta$-observable inscribed radius $\OI(X;-\eta)$ of $X$} by
\begin{equation*}
\OI(X;-\eta):=\sup_{\varphi}\,\PI(I_{\varphi};1-\eta),
\end{equation*}
where the supremum is taken over all $1$-Lipschitz functions $\varphi:X\to I$ with $\varphi|_{\partial X}=0$.
\end{defi}
We remark that
$\OI(X;-\eta)=0$ for all $\eta \geq 1$.
Furthermore,
$\OI(X;-\eta)$ is monotone non-increasing in $\eta$.

We also introduce the following notion:
\begin{defi}\label{defi:Dirichlet Levy family}
We say that
a sequence $\{X_{n}\}$ of metric measure spaces with boundary is a \textit{boundary concentration family} if
for every $\eta>0$
\begin{equation*}
\lim_{n\to \infty} \, \OI \left(X_{n};-\eta \right)=0.
\end{equation*}
\end{defi}

\begin{rem}\label{rem:characterization of Dirichlet Levy}
Let $\{X_{n}\}$ denote a sequence of metric measure spaces with boundary $X_{n}=(X_{n},d_{X_{n}},\mu_{X_{n}})$.
By the definition of the observable inscribed radius,
$\{X_{n}\}$ is a boundary concentration family if and only if
for every sequence $\{\varphi_{n}\}$ of $1$-Lipschitz functions $\varphi_{n}:X_{n}\to I$ with $\varphi_{n}|_{\partial X_{n}}=0$,
we have $d_{KF}(\varphi_{n},0) \to 0$ as $n\to \infty$,
where $d_{KF}(\varphi_{n},0)$ is the \textit{Ky Fan metric} between $\varphi_{n}$ and $0$ defined as
\begin{equation*}
d_{KF}(\varphi_{n},0):=\inf \left\{\, \epsilon \geq 0 \mid \mu_{X_{n}} \left(  \left\{\, x\in X_{n} \mid \varphi_{n}(x)>\epsilon \right\} \right) \leq \epsilon \,\right\}.
\end{equation*}
We further see that the following are equivalent (see Proposition \ref{prop:characterization of Dirichlet Levy}):
\begin{enumerate}\setlength{\itemsep}{+0.7mm}
\item $\{X_{n}\}$ is a boundary concentration family;
\item for every sequence $\{\Omega_{n}\}$ of Borel subsets $\Omega_{n} \subset X_{n}$ satisfying $\liminf_{n\to \infty} \mu_{X_{n}}(\Omega_{n})>0$,
         we have $\lim_{n\to \infty}\, d_{X_{n}}(\Omega_{n},\partial X_{n})=0$;
\item $\lim_{n\to \infty}\,\mu_{X_{n}}(B_{r}(\partial X_{n}))=1$ for every $r>0$,
         where $B_{r}(\partial X_{n})$ is the closed $r$-neighborhood of $\partial X_{n}$.
\end{enumerate}
\end{rem}

%%%%%%%%%%%%%%%%%%%%%%%%
\subsection{Dirichlet eigenvalues and concentration phenomena}
We study relation between boundary concentration phenomena and large spectral gap phenomena for Dirichlet eigenvalues of Laplacian.
For $n\geq 2$,
let $(M,d_{M},m_{M,f})$ be an $n$-dimensional weighted Riemannian manifold with boundary,
namely,
$M=(M,g)$ be an $n$-dimensional,
connected complete Riemannian manifold with boundary,
$d_{M}$ is the Riemannian distance on $M$,
and
\begin{equation}\label{eq:weighted volume measure}
m_{M,f}:=e^{-f} \,\vol_{M}
\end{equation}
for some smooth function $f:M\to \mathbb{R}$,
where $\vol_{M}$ is the Riemannian volume measure on $M$.
The \textit{weighted Laplacian $\Delta_{f}$} is defined by
\begin{equation}\label{eq:weighted Laplacian}
\Delta_{f}:=\Delta+g(\nabla f,\nabla \cdot),
\end{equation}
where $\nabla$ is the gradient,
and $\Delta$ is the Laplacian defined as the minus of the trace of Hessian.
In the case where $M$ is compact,
we consider the following Dirichlet eigenvalue problem with respect to $\Delta_{f}$:
\begin{equation*}
\begin{cases}
  \Delta_{f}\, \phi=\nu\, \phi & \text{in $\inte M$};\\
  \phi=0                        & \text{on $\bm$},
\end{cases}
\end{equation*}
where $\inte M$ denotes the interior of $M$.
We denote by
\begin{equation}\label{eq:notation of Dirichlet eigenvalue}
0< \nu_{f,1}(M) < \nu_{f,2}(M) \leq \dots \leq \nu_{f,k}(M)\leq \dots \nearrow +\infty
\end{equation}
the all Dirichlet eigenvalues of $\Delta_{f}$,
counting multiplicity.

For a smooth function $f:M\to \mathbb{R}$ such that $m_{M,f}$ is a Borel probability measure,
we study the metric measure space with boundary
\begin{equation}\label{eq:weighted Riemannian manifold}
(M,f):=(M,d_{M},m_{M,f}).
\end{equation}
Our first main result is the following:
\begin{thm}\label{thm:Dirichlet eigenvalue and Dirichlet Levy family}
Let $\{(M_{n},f_{n})\}$ be a sequence of compact metric measure spaces with boundary defined as $(\ref{eq:weighted Riemannian manifold})$.
If we have $\nu_{f_{n},1}(M_{n})\to \infty$ as $n\to \infty$,
then $\{(M_{n},f_{n})\}$ is a boundary concentration family.
\end{thm}
This is an analogue of the Gromov-V. D. Milman theorem for closed manifolds (compact manifolds without boundary) (see Theorem 4.1 and its corollary in \cite{GM}).
We show Theorem \ref{thm:Dirichlet eigenvalue and Dirichlet Levy family} by using relation between the observable inscribed radii and $\nu_{f,1}(M)$ (see Proposition \ref{prop:logarithmic Dirichlet eigenvalue and observable inscribed radius}).

For higher eigenvalues,
we will establish the following assertion under $\ric^{\infty}_{f,M}\geq 0$ and $H_{f,\bm}\geq 0$,
where $\ric^{\infty}_{f,M}$ and $H_{f,\bm}$ are the infimum of the $\infty$-weighted Ricci curvature and the weighted mean curvature on $M$ and on $\bm$,
respectively (more precisely, see Subsection \ref{sec:Curvatures}):
\begin{thm}\label{thm:higher order Dirichlet eigenvalue and Dirichlet Levy family}
Let $\{(M_{n},f_{n})\}$ be a sequence of compact metric measure spaces with boundary defined as $(\ref{eq:weighted Riemannian manifold})$.
Assume that
$\ric^{\infty}_{f_{n},M_{n}}\geq 0$ and $H_{f_{n},\bm_{n}}\geq 0$.
If there exists $k\geq 1$ such that $\nu_{f_{n},k}(M_{n})\to \infty$ as $n\to \infty$,
then $\{(M_{n},f_{n})\}$ is a boundary concentration family.
\end{thm}

Theorem \ref{thm:higher order Dirichlet eigenvalue and Dirichlet Levy family} is an analogue of the Funano-Shioya theorem for closed manifolds of non-negative $\infty$-weighted Ricci curvature (see Corollary 1.4 in \cite{FSh}).
One of key ingredients of the proof is to obtain an upper bound of the ratio $\nu_{f,k}(M)/\nu_{f,1}(M)$ in terms of $C\,k^{2}$ for some universal constant $C>0$ under $\ric^{\infty}_{f,M}\geq 0$ and $H_{f,\bm}\geq 0$ (see Theorem \ref{thm:universal inequality}).
We obtain such universal estimate by combining an improved Cheeger inequality of Kwak, Lau, Lee, Oveis Gharan and Trevisan \cite{KLLOT},
and an isoperimetric inequality of Wang \cite{Wa} (see Subsections \ref{sec:Dirichlet isoperimetric constants}, \ref{sec:Isoperimetric inequalities}).

%%%%%%%%%%%%%%%%%%%%%%%%
\subsection{Comparisons and concentration phenomena}\label{sec:Comparisons and concentration phenomena}
To understand boundary concentration phenomena,
we establish comparison theorems of the observable inscribed radii under a lower curvature Ricci curvature bound,
and a lower mean curvature bound for the boundary.

We first present finite dimensional comparisons.
Let $M$ be an $n$-dimensional,
connected complete Riemannian manifold with boundary with $\vol_{M}(M)<\infty$.
We study metric measure space with boundary
\begin{equation}\label{eq:normalized metric measure space with boundary}
M:=(M,d_{M},m_{M}),\quad m_{M}:=\frac{1}{\vol_{M} (M)} \,\vol_{M}.
\end{equation}
Let $\ric_{\bm^{\perp}}$ stand for the infimum of the Ricci curvature in the $\bm$-radial direction on $M$,
and $H_{\bm}$ the infimum of the mean curvature on $\bm$ (more precisely, see Subsection \ref{sec:Curvatures}).

For $\kappa \in \mathbb{R}$,
let $M^{n}_{\kappa}$ be the $n$-dimensional space form with constant curvature $\kappa$.
For $\lambda \in \mathbb{R}$,
we say that $\kappa$ and $\lambda$ satisfy the \textit{ball-condition}
if there exists a closed ball $\ball$ in $M^{n}_{\kappa}$ whose boundary has constant mean curvature $(n-1)\lambda$.
Let $\const$ denote its radius.
Note that $\kappa$ and $\lambda$ satisfy the ball-condition if and only if either
(1) $\kappa>0$; 
(2) $\kappa=0$ and $\lambda>0$;
or (3) $\kappa<0$ and $\lambda>\sqrt{\vert \kappa \vert}$.
We say that $\kappa$ and $\lambda$ satisfy the \textit{convex-ball-condition}
if they satisfy the ball-condition and $\lambda \geq 0$.

Let us prepare the following finite dimensional model spaces: 
(1) For $\kappa$ and $\lambda$ satisfying the ball-condition,
we call the metric measure space with boundary
\begin{equation}\label{eq:ball model}
\ball=\bigl(\ball,d_{\ball},m_{\ball} \bigl)
\end{equation}
the \textit{ball-model-space};
(2) For $\kappa<0$ and $\lambda:=\sqrt{\vert \kappa \vert}$,
we consider the warped product space $M^{n}_{\kappa,\lambda}:=([0,\infty) \times \mathbb{S}^{n-1},dt^{2}+e^{-2\lambda t}\,ds_{n-1}^{2})$,
where $(\mathbb{S}^{n-1},ds_{n-1}^{2})$ is the $(n-1)$-dimensional standard unit sphere.
We call the metric measure space with boundary
\begin{equation}\label{eq:warped product model}
M^{n}_{\kappa,\lambda}=\bigl(M^{n}_{\kappa,\lambda},d_{M^{n}_{\kappa,\lambda}},m_{M^{n}_{\kappa,\lambda}}\bigl)
\end{equation}
the \textit{warped-product-model-space} (cf. Remark \ref{rem:formulation of finite dimensional model spaces}).

We have the following finite dimensional comparison theorem:
\begin{thm}\label{thm:main theorem}
Let $\bm$ be compact.
We assume $\ric_{\bm^{\perp}}\geq (n-1)\kappa$ and $H_{\bm}\geq (n-1)\lambda$.
Then for every $\eta \in (0,1]$
the following hold:
\begin{enumerate}
\item if $\kappa$ and $\lambda$ satisfy the ball-condition,
         then
         \begin{equation*}
         \OI(M;-\eta) \leq \OI(\ball;-\eta);
         \end{equation*}
\item if $\kappa<0$ and $\lambda=\sqrt{\vert \kappa \vert}$,
         then
         \begin{equation*}
         \OI(M;-\eta) \leq \OI(M^{n}_{\kappa,\lambda};-\eta).
         \end{equation*}
\end{enumerate}
\end{thm}
\begin{rem}\label{rem:finite volume}
For (1) of Theorem \ref{thm:main theorem},
we always have $\vol_{M} (M)<\infty$.
Indeed,
the Heintze-Karcher theorem (Theorem 2.1 in \cite{HK}) leads to
\begin{equation*}
\frac{\vol_{M} (M)}{\vol_{\bm} (\bm)}\leq \frac{\vol_{\ball} (\ball)}{\vol_{\partial \ball} (\partial \ball)}.
\end{equation*}
Similarly,
for (2),
the Heintze-Karcher theorem guarantees that
the Riemannian volume $\vol_{M} (M)$ of $M$ is finite.
\end{rem}

The proof of Theorem \ref{thm:main theorem} is based on comparison geometry of manifolds with boundary
established by Heintze and Karcher \cite{HK}, Kasue \cite{K1}, \cite{K2}, the author \cite{S1}, \cite{S2}, and so on (see Subsection \ref{sec:Comparison theorems}).
We first estimate observable inscribed radii under lower weighted curvature bounds by using relative volume comparison theorems for metric neighborhoods of boundaries (see Theorems \ref{thm:weighted main theorem} and \ref{thm:twisted weighted main theorem}).
We conclude Theorem \ref{thm:main theorem} by computing the observable inscribed radii of finite dimensional model spaces (see Lemma \ref{lem:observable and separation on model}).

We next produce infinite dimensional comparisons for metric measure spaces with boundary $(M,f)$ defined as (\ref{eq:weighted Riemannian manifold}) under $\ric^{\infty}_{f,\bm^{\perp}}\geq K$ and $H_{f,\bm}\geq \Lambda$ for $K,\Lambda\in \mathbb{R}$,
where $\ric^{\infty}_{f,\bm^{\perp}}$ is the infimum of the $\infty$-weighted Ricci curvature in the $\bm$-radial direction on $M$ (see Subsection \ref{sec:Curvatures}).
We prepare the following infinite dimensional model spaces: 
(1) For $K>0,\,\Lambda \in \mathbb{R}$,
we call the metric measure space with boundary
\begin{equation}\label{eq:Gaussian model space}
G_{K,\Lambda}:=\left(I,d_{I}, \frac{e^{-\frac{K}{2}\,t^{2}-\Lambda\,t}}{\int_{I}\,e^{-\frac{K}{2}\,t^{2}-\Lambda\,t}\,    dt} \vol_{I} \right)
\end{equation}
the \textit{half-Gaussian-model-space};
(2) For $\Lambda>0$,
we call
\begin{equation}\label{eq:exponential model space}
E_{\Lambda}:=\left(I,d_{I}, \Lambda\,e^{-\Lambda\,t} \vol_{I} \right)
\end{equation}
the \textit{exponential-model-space}.
We remark that
for $K,\Lambda\in \mathbb{R}$,
the value $\int_{I}\,e^{-\frac{K}{2}\,t^{2}-\Lambda\,t}\,    dt$ is finite if and only if
either (1) $K>0$;
or (2) $K=0$ and $\Lambda>0$.
Moreover,
if $K=0$ and $\Lambda>0$,
then we see
\begin{equation*}
\frac{e^{-\frac{K}{2}\,t^{2}-\Lambda\,t}}{\int_{I}\,e^{-\frac{K}{2}\,t^{2}-\Lambda\,t}\,    dt}=\Lambda\,e^{-\Lambda\,t}.
\end{equation*}
We observe that
our infinite dimensional model spaces appear as limits of sequences of finite dimensional model spaces (see Subsection \ref{sec:Infinite dimensional model spaces and distribution laws}).

We have the following infinite dimensional comparison:
\begin{thm}\label{thm:infinite dimensional main theorem}
Let $\bm$ be compact.
Let us assume $\ric^{\infty}_{f,\bm^{\perp}}\geq K$ and $H_{f,\bm}\geq \Lambda$.
Then for every $\eta \in (0,1]$
the following hold:
\begin{enumerate}
\item if $K>0$ and $\Lambda \in \mathbb{R}$,
         then
         \begin{equation*}
         \OI((M,f);-\eta) \leq \OI(G_{K,\Lambda};-\eta);
         \end{equation*}
\item if $K=0$ and $\Lambda>0$,
         then
         \begin{equation*}
         \OI((M,f);-\eta) \leq \OI(E_{\Lambda};-\eta).
         \end{equation*}
\end{enumerate}
\end{thm}
To prove Theorem \ref{thm:infinite dimensional main theorem},
we develop comparison geometry of manifolds with boundary under $\ric^{\infty}_{f,\bm^{\perp}}\geq K$ and $H_{f,\bm}\geq \Lambda$ for $K,\Lambda\in \mathbb{R}$ (see Subsection \ref{sec:Relative volume comparisons}).
We will show a relative comparison theorem for metric neighborhoods of boundaries (see Theorem \ref{thm:relative volume comparison}).

Having Theorems \ref{thm:main theorem} and \ref{thm:infinite dimensional main theorem} at hand,
we will study various boundary concentration phenomena of sequences of metric measure spaces with boundary (see Section \ref{sec:Boundary concentration phenomena}).
For instance,
for a sequence of ball-model-spaces,
we conclude the following:
\begin{cor}\label{cor:ball model Levy family}
If $\kappa$ and $\lambda$ satisfy the convex-ball-condition,
then the sequence $\{\ball\}$ is a boundary concentration family.
\end{cor}
\begin{rem}
In the case where $\kappa>0$ and $\lambda<0$,
the sequence $\{\ball\}$ is not a boundary concentration family.
Indeed,
for $r\in (0,\const-C_{\kappa,0})$,
if we define $\Omega_{n} \subset \ball$ as the $r/2$-neighborhood of the metric sphere with same center as $\ball$ and radius $C_{\kappa,0}$,
then $\liminf_{n\to \infty} m_{\ball}(\Omega_{n})>0$ and $\lim_{n\to \infty} d_{\ball}(\Omega_{n},\partial \ball)>0$ (see Remark \ref{rem:characterization of Dirichlet Levy} and Proposition \ref{prop:characterization of Dirichlet Levy}).
\end{rem}

%%%%%%%%%%%%%%%%%%%%%%%%%%%%%%%%
\subsection{Organization}
In Section \ref{sec:Preliminaries},
we review basics of weighted Riemannian manifolds with boundary,
and examine their geometric and analytic properties.
In Section \ref{sec:Invariants},
we introduce some invariants on metric measure spaces with boundary,
and investigate their fundamental properties.
In Section \ref{sec:Dirichlet eigenvalues},
we prove Theorems \ref{thm:Dirichlet eigenvalue and Dirichlet Levy family} and \ref{thm:higher order Dirichlet eigenvalue and Dirichlet Levy family}.
In Section \ref{sec:Comparisons},
we prove Theorem \ref{thm:main theorem}.
In Section \ref{sec:Infinite dimensional comparisons},
we prove Theorem \ref{thm:infinite dimensional main theorem}.

Section \ref{sec:Boundary concentration phenomena} is devoted to the collection of boundary concentration phenomena of sequences of metric measure spaces with boundary.
We will determine the critical scale orders of some sequences of finite dimensional model spaces (see Theorems \ref{thm:sequence of hemisphere}, \ref{thm:sequence of Euclidean ball}, \ref{thm:warped product model Levy family}).
We also prove Corollary \ref{cor:ball model Levy family}.
Furthermore,
we construct several non-trivial examples of boundary concentration families (see Examples \ref{ex:finite dimensional example} and \ref{ex:one dimensional example}).

%%%%%%%%%%%%%%%%%%%%%%%%%%%%%%%%
\subsection*{{\rm Acknowledgements}}
The author expresses his gratitude to Ryunosuke Ozawa and Kohei Suzuki for fruitful discussions and encouragement.
The author would like to thank Shin-ichi Ohta for his useful comments.
One of his comment leads to the study of Section \ref{sec:Infinite dimensional comparisons}.
The author would also like to thank Daisuke Kazukawa for informing him of Proposition \ref{prop:Kazukawa} with its proof.
The author is deeply grateful to Kei Funano for his interests in this paper,
valuable comments,
and informing him of \cite{CGY}.

\section{Preliminaries}\label{sec:Preliminaries}
Throughout this section,
let $(M,d_{M},m_{M,f})$ denote an $n$-dimensional weighted Riemannian manifold with boundary defined as (\ref{eq:weighted volume measure}). 

%%%%%%%%%%%%%%%%%%%%%%%%
\subsection{Curvatures}\label{sec:Curvatures}
We denote by $\ric_{g}$ the Ricci curvature on $M$ determined by the Riemannian metric $g$,
and by $\ric_{M}$ the infimum of $\ric_{g}$ on the unit tangent bundle over $M$.
For $N\in (-\infty,\infty]$,
the \textit{$N$-weighted Ricci curvature} $\ric^{N}_{f}$ is defined as follows:
If $N \in (-\infty,\infty)\setminus \{n\}$,
then
\begin{equation*}
\ric^{N}_{f}:=\ric_{g}+\Hess f-\frac{d f \otimes d f}{N-n},
\end{equation*}
where $d f$ and $\Hess f$ are the differential and the Hessian of $f$,
respectively;
otherwise,
if $N=\infty$,
then $\ric^{N}_{f}:=\ric_{g}+\Hess f$;
if $N=n$,
and if $f$ is constant,
then $\ric^{N}_{f}:=\ric_{g}$;
if $N=n$,
and if $f$ is not constant,
then $\ric^{N}_{f}:=-\infty$ (\cite{BE}, \cite{L}).
For a function $\mathcal{F}:M\to \mathbb{R}$,
we mean by $\ric^{N}_{f,M}\geq \mathcal{F}$ for every $x \in M$,
and for every unit tangent vector $v$ at $x$
it holds that $\ric^{N}_{f}(v)\geq \mathcal{F}(x)$.

\begin{rem}\label{rem:monotonicity of weighted Ricci curvature}
Traditionally,
$N$ has been chosen from $[n,\infty]$ (see e.g., \cite{Lo}, \cite{Q}, \cite{WW}).
On the other hand,
recently,
various properties have begun to be studied in the complemental case of $N \in (-\infty,n)$ (see e.g., \cite{K}, \cite{KM}, \cite{Mi4}, \cite{O1}, \cite{O2}, \cite{OT1}, \cite{OT2}, \cite{W}, \cite{WY}).
We notice the monotonicity of $\ric^{N}_{f}$ with respect to $N$:
If $N_{1},\,N_{2} \in [n,\infty]$ with $N_{1}\leq N_{2}$,
then $\ric^{N_{1}}_{f} \leq \ric^{N_{2}}_{f}$;
if $N_{1} \in [n,\infty]$ and $N_{2} \in (-\infty,n)$,
then $\ric^{N_{1}}_{f} \leq \ric^{N_{2}}_{f}$;
if $N_{1},\,N_{2} \in (-\infty,n)$ with $N_{1}\leq N_{2}$,
then $\ric^{N_{1}}_{f} \leq \ric^{N_{2}}_{f}$.
\end{rem}

For $z\in \bm$,
let $u_{z}$ denote the unit inner normal vector for $\bm$ at $z$,
and let $\gamma_{z}:[0,T)\to M$ denote the geodesic with $\gamma_{z}'(0)=u_{z}$.
We put
\begin{equation*}
\tau(z):=\sup \{\,t>0 \mid \rho_{\bm}(\gamma_{z}(t))=t\,\}.
\end{equation*}
Let $\ric_{\bm^{\perp}}$ be the \textit{infimum of the Ricci curvature in the $\bm$-radial direction on $M$} defined as $\inf_{z,t} \ric_{g}(\gamma'_{z}(t))$,
where the infimum is take over all $z\in \bm,\,t \in (0,\tau(z))$.
For $\mathcal{F}:M\to \mathbb{R}$,
we mean by $\ric^{N}_{f,\bm^{\perp}} \geq \mathcal{F}$ for all $z \in \bm,\,t \in (0,\tau(z))$
we have $\ric^{N}_{f}(\gamma'_{z}(t))\geq \mathcal{F}(\gamma_{z}(t))$.

For vector fields $v_{1},\,v_{2}$ on $\bm$,
the second fundamental form $S(v_{1},v_{2})$ is defined as the normal component of $\nabla^{g}_{v_{1}}v_{2}$ with respect to $\bm$,
where $\nabla^{g}$ denotes the Levi-Civita connection induced from $g$.
For $z\in \bm$,
let $T_{z}\bm$ denote the tangent space at $z$ on $\bm$.
The shape operator $A_{u_{z}}:T_{z}\bm \to T_{z}\bm$ for $u_{z}$ is defined as
\begin{equation*}
g(A_{u_{z}}v_{1},v_{2}):=g(S(v_{1},v_{2}),u_{z}).
\end{equation*}
The mean curvature $H_{z}$ at $z$ is defined as the trace of $A_{u_{z}}$.
Put $H_{\bm}:=\inf_{z\in \bm} H_{z}$.
The \textit{weighted mean curvature} $H_{f,z}$ at $z$ is defined by
\begin{equation*}
H_{f,z}:=H_{z}+g((\nabla f)_{z},u_{z}).
\end{equation*}
For a function $\mathcal{G}:\bm\to \mathbb{R}$,
we mean by $H_{f,\bm}\geq \mathcal{G}$
we have $H_{f,z} \geq \mathcal{G}(z)$ for all $z\in \bm$.

We mainly study the following three curvature conditions:
For $\kappa,\lambda \in \mathbb{R}$ and $K,\Lambda \in \mathbb{R}$,
\begin{align}\label{eq:finite dimensional curvature bound}
N \in [n,\infty),\,\, &\ric^{N}_{f,\bm^{\perp}}\geq (N-1)\kappa,\,\, H_{f,\bm}\geq (N-1)\lambda;\\ \label{eq:infinite dimensional curvature bound}
N=\infty,\,\,&\ric^{\infty}_{f,\bm^{\perp}}\geq K,\, H_{f,\bm}\geq \Lambda; \\ \label{eq:one dimensional curvature bound}
N=1,\,\,&\ric^{1}_{f,\bm^{\perp}}\geq (n-1)\kappa\, e^{\frac{-4f}{n-1}},\,\, H_{f,\bm}\geq (n-1)\lambda\, e^{\frac{-2f}{n-1}}.
\end{align}

\begin{rem}
We give a historical comment for the curvature condition (\ref{eq:one dimensional curvature bound}).
First,
Wylie \cite{W} has obtained a splitting theorem of Cheeger-Gromoll type under $\ric^{N}_{f,M}\geq 0$ for $N\in (-\infty,1]$.
After that
Wylie and Yeroshkin \cite{WY} have introduced the condition $\ric^{1}_{f,M}\geq (n-1)\kappa\, e^{\frac{-4f}{n-1}}$ from the view point of study of affine connections,
and established comparison geometry.
Furthermore,
the author \cite{S2} has studied comparison geometry of manifolds with boundary under the curvature condition $\ric^{N}_{f,M}\geq (n-1)\kappa\, e^{\frac{-4f}{n-1}}$ and $H_{f,\bm}\geq (n-1)\lambda\, e^{\frac{-2f}{n-1}}$ for $N\in (-\infty,1]$.
\end{rem}

%%%%%%%%%%%%%%%%%%%%%%%%
\subsection{Laplacians and Dirichlet eigenvalues}
For the weighted Laplacian $\Delta_{f}$ defined as (\ref{eq:weighted Laplacian}),
the following formula of Bochner type is well-known (see e.g., Chapter 14 in \cite{V}):
For every smooth $\psi:M\to \mathbb{R}$,
\begin{equation}\label{eq:Bochner formula}
-\frac{1}{2}\,\Delta_{f}\,\Vert \nabla \psi \Vert^{2}=\ric^{\infty}_{f}(\nabla \psi)+\Vert \Hess \psi \Vert^{2}_{\HS}-g\left(\nabla \Delta_{f}\,\psi,\nabla \psi \right),
\end{equation}
where $\Vert \cdot \Vert$ and $\Vert \cdot \Vert_{\HS}$ are the canonical norm and the Hilbert-Schmidt norm induced from $g$,
respectively.

For $z\in \bm$,
the value $\Delta_{f}\rho_{\bm}(\gamma_{z}(t))$ converges to $H_{f,z}$ as $t\to 0$.
For $t \in (0,\tau(z))$,
and for the volume element $\theta(t,z)$ of the $t$-level surface of $\rho_{\bm}$ at $\gamma_{z}(t)$,
we set
\begin{equation}\label{eq:volume element}
\theta_{f}(t,z):=e^{-f(\gamma_{z}(t))}\, \theta(t,z).
\end{equation}
For all $t\in (0,\tau(z))$
it holds that
\begin{equation}\label{eq:Laplacian representation}
\Delta_{f}\, \rho_{\bm}(\gamma_{z}(t)) =-\frac{\theta_{f}'(t,z)}{\theta_{f}(t,z)}.
\end{equation}
We also have the following (see e.g., \cite{S1}):
If $\bm$ is compact,
then
\begin{equation}\label{eq:integration formula}
m_{M,f}( B_{r}(\bm))=\int_{\bm} \int^{r}_{0}\bar{\theta}_{f}(t,z)\,dt\,d\vol_{h}
\end{equation}
for all $r>0$,
where $\vol_{h}$ is the Riemannian volume measure on $\bm$ induced from $h$,
and $\bar{\theta}_{f}:[0,\infty) \times \bm \to \mathbb{R}$ is a function defined as
\begin{equation}\label{eq:extended volume element}
  \bar{\theta}_{f}(t,z) := \begin{cases}
                                     \theta_{f}(t,z) & \text{if $t< \tau(z)$}, \\
                                             0           & \text{if $t \geq \tau(z)$}.
                                    \end{cases}
\end{equation}

Let $M$ be compact.
For $\phi \in H^{1}_{0}(M,m_{M,f})\setminus \{0\}$,
its Rayleigh quotient is defined as
\begin{equation}\label{eq:Rayleigh quotient}
R_{f}(\phi):=\frac{\int_{M}\, \Vert \nabla \phi \Vert^{2}\,d\,m_{M,f}}{\int_{M}\, \phi^{2}\,d\,m_{M,f}}.
\end{equation}
For the $k$-th Dirichlet eigenvalue $\nu_{f,k}(M)$ of the weighted Laplacian $\Delta_{f}$ defined as $(\ref{eq:notation of Dirichlet eigenvalue})$,
the min-max principle states
\begin{equation}\label{eq:min-max principle}
\nu_{f,k}(M)=\inf_{L} \sup_{\phi \in L\setminus \{0\}} R_{f}(\phi),
\end{equation}
where the infimum is taken over all $k$-dimensional subspaces $L$ of the Sobolev space $H^{1}_{0}(M,m_{M,f})$.

%%%%%%%%%%%%%%%%%%%
\subsection{Dirichlet isoperimetric constants}\label{sec:Dirichlet isoperimetric constants}
For a Borel subset $\Omega \subset M$,
\begin{equation*}
m^{+}_{M,f}(\Omega):=\liminf_{r\to 0}\frac{m_{M,f}(U_{r}(\Omega))-m_{M,f}(\Omega)}{r},
\end{equation*}
where $U_{r}(\Omega)$ is the open $r$-neighborhood of $\Omega$.
The \textit{Dirichlet isoperimetric constant} is defined as
\begin{equation*}
\mathcal{I}_{f}(M):=\inf_{\Omega} \frac{m^{+}_{M,f}(\Omega)}{m_{M,f}(\Omega)},
\end{equation*}
where the infimum is taken over all Borel subsets $\Omega \subset \inte M$ (cf. Definition 9.1 in \cite{L}).
The following inequality of Cheeger type is well-known (see \cite{C}, and cf. Corollary 9.7 in \cite{L}):
If $M$ is compact,
then
\begin{equation}\label{eq:Dirichlet Cheeger inequality}
\mathcal{I}_{f}(M) \leq 2 \sqrt{\nu_{f,1}(M)}.
\end{equation}

In the graph setting,
Kwak, Lau, Lee, Oveis Gharan and Trevisan \cite{KLLOT} have established an improved Cheeger inequality in terms of the smallest and higher eigenvalues of the Laplacian and the conductance (see Theorem 1.1 in \cite{KLLOT}).
In the manifold setting,
to answer a question of Funano \cite{F},
Liu \cite{Liu} has pointed out that
a similar improved Cheeger inequality holds for closed eigenvalues of the Laplacian and the Cheeger constant via the same argument as in \cite{KLLOT} (see Theorem 1.6 in \cite{Liu}).

Now,
we further point out that
the following improvement of (\ref{eq:Dirichlet Cheeger inequality}) holds in our setting via the same argument as in \cite{KLLOT}:
\begin{thm}[\cite{KLLOT}, \cite{Liu}]\label{thm:higer order eigenvalue and isoperimetric}
Let $M$ be compact.
Then for all $k\geq 1$,
\begin{equation}\label{eq:higer order eigenvalue and isoperimetric}
\mathcal{I}_{f}(M) \leq 8\sqrt{2}k\frac{\nu_{f,1}(M)}{\sqrt{\nu_{f,k}(M)}}.
\end{equation}
\end{thm}
One can verify (\ref{eq:higer order eigenvalue and isoperimetric}) by applying the same argument as in the proof of Lemmas 3.2 and 3.3 in \cite{Liu} (replace the role of the $k$-th closed eigenvalue with that of $\nu_{f,k}(M)$) to a non-negative eigenfunction of $\nu_{f,1}(M)$,
here we recall that any eigenfunctions of $\nu_{f,1}(M)$ are either always positive or always negative on $\inte M$.
Note that
for such an eigenfunction,
the $2k$ disjointly supported Lipschitz functions constructed in the proof of Lemma 3.3 in \cite{Liu} also satisfy the Dirichlet boundary condition.

%%%%%%%%%%%%%%%%%%%
\subsection{Dirichlet eigenvalue estimates}\label{sec:Isoperimetric inequalities}
Wang \cite{Wa} has produced a gradient estimate of Bakry-Ledoux type for the Dirichlet heat semigroup associated with the weighted Laplacian under a lower (unweighted) Ricci curvature bound,
a lower (unweighted) mean curvature bound for the boundary,
and a density bound (see Theorem 1.1 in \cite{Wa}, and also \cite{BL}).
From the same argument as in the proof of Theorem 1.1 in \cite{Wa},
we can derive the following (cf. (1.10) in \cite{Wa} for unweighted case):
\begin{thm}[\cite{Wa}]\label{thm:gradient estimate}
If $\ric^{\infty}_{f,M}\geq 0$ and $H_{f,\bm}\geq 0$,
then we have
\begin{equation*}
\Vert \Vert \nabla P_{t}\psi \Vert \Vert_{L^{\infty}}\leq \frac{\sqrt{1+2^{1/3}}\,(1+4^{2/3})}{2\sqrt{\pi}} \frac{\Vert \psi \Vert_{L^{\infty}}}{\sqrt{t}}
\end{equation*}
for all $t>0$ and non-negative, bounded measurable functions $\psi$ on $M$,
where $P_{t}$ is the Dirichlet heat semigroup generated by $-\Delta_{f}$.
\end{thm}
Wang \cite{Wa} has proved Theorem \ref{thm:gradient estimate} when $f=0$.
One can see Theorem \ref{thm:gradient estimate} only by using Lemma 3.4 in \cite{S1} (or Lemma \ref{lem:Laplacian comparison} below) instead of Lemma 2.3 in \cite{Wa},
and using the inequality (2.3) in \cite{CW} instead of (2.5) in \cite{Wa} along the line of the proof of Theorem 1.1 in \cite{Wa}.

In virtue of the gradient estimate,
Wang \cite{Wa} has obtained an isoperimetric inequality of Buser and Ledoux type based on the idea of Ledoux \cite{L1} (see Theorem 1.2 in \cite{Wa}, and also \cite{Bu}, \cite{L1}).
Theorem \ref{thm:gradient estimate} together with the same argument as in the proof of Theorem 1.2 in \cite{Wa} yields:
\begin{thm}[\cite{Wa}]\label{thm:Buser Ledoux}
Let $M$ be compact.
If $\ric^{\infty}_{f,M}\geq 0$ and $H_{f,\bm}\geq 0$,
then we have
\begin{equation}\label{eq:Buser Ledoux}
\mathcal{I}_{f}(M) \geq \frac{2\sqrt{\pi}}{\sqrt{1+2^{1/3}}\,(1+4^{2/3})}\,\sup_{t>0}\frac{1-e^{-t}}{\sqrt{t}}\, \sqrt{\nu_{f,1}(M)}.
\end{equation}
\end{thm}

Combining (\ref{eq:higer order eigenvalue and isoperimetric}), (\ref{eq:Buser Ledoux}) implies the following universal inequality:
\begin{thm}\label{thm:universal inequality}
Let $M$ be compact.
If $\ric^{\infty}_{f,M}\geq 0$ and $H_{f,\bm}\geq 0$,
then there is a universal constant $C>0$ such that
for all $k\geq 1$
we have
\begin{equation}\label{eq:universal inequality}
\nu_{f,k}(M)\leq C\,k^{2}\,\nu_{f,1}(M).
\end{equation}
\end{thm}
Funano and Shioya \cite{FSh}, Funano \cite{F}, Liu \cite{Liu} have formulated similar inequalities for closed manifolds of non-negative $\infty$-weighted Ricci curvature (see Theorem 1.1 in \cite{FSh}, Theorem 1.2 in \cite{F} and Theorem 1.1 in \cite{Liu}).
The inequality (\ref{eq:universal inequality}) corresponds to that of Liu \cite{Liu}.

\begin{rem}\label{rem:Li-Yau dimension free}
Under similar setting $\ric^{\infty}_{f,\bm^{\perp}}\geq 0$ and $H_{f,\bm}\geq 0$ to that in Theorem \ref{thm:universal inequality},
the author \cite{S1} has shown a dimension free inequality
\begin{equation}\label{eq:Li-Yau dimension free}
\nu_{f,1}(M)\geq \pi^{2} (2\IR M)^{-2}
\end{equation}
of Li-Yau, Kasue type,
and a rigidity result for the equality case (see Corollary 7.6 in \cite{S1}, and also \cite{LY}, \cite{K2}, and cf. Remarks \ref{rem:remark on the curvature assumption} and \ref{rem:conjecture}).
\end{rem}

\begin{rem}\label{rem:dimension dependent inequality}
Let $M$ be compact,
and let $\nu_{k}(M)$ be the $k$-th Dirichlet eigenvalue of the Laplacian $\Delta$.
Let us mention a dimension dependent estimate of the ratio $\nu_{k}(M)/\nu_{1}(M)$ induced from a classical method by Cheng \cite{Ch}, Li-Yau \cite{LY}.
We possess the following estimate by modifying the proof of Corollary 2.2 in \cite{Ch} (take a unit speed minimal geodesic $\gamma:[0,\IR M]\to M$ with $\gamma((0,\IR M])\subset \inte M$ that is orthogonal to $\bm$ at $\gamma(0)$,
set $k$ disjoint open balls in $\inte M$ centered at $\gamma((2\alpha-1)(2k)^{-1} \IR M)$ with radius $(2k)^{-1} \IR M$ for $\alpha=1,\dots,k$,
and apply the argument of proof of Theorem 2.1 in \cite{Ch}):
If $\ric_{M}\geq 0$,
then
\begin{equation}\label{eq:Cheng dimension dependent}
\nu_{k}(M)\leq 2n(n+4)k^{2} (\IR M)^{-2}
\end{equation}
for all $k\geq 1$.
By (\ref{eq:Li-Yau dimension free}), (\ref{eq:Cheng dimension dependent}),
we obtain the following:
If $\ric_{M}\geq 0$ and $H_{\bm}\geq 0$,
then there is $C_{n}>0$ depending only on $n$ such that
\begin{equation}\label{eq:Li-Yau-Cheng dimension dependent inequality}
\nu_{k}(M)\leq C_{n}\,k^{2}\,\nu_{1}(M)
\end{equation}
for all $k\geq 1$.
Theorem \ref{thm:universal inequality} is a refinement of (\ref{eq:Li-Yau-Cheng dimension dependent inequality}) in the sense that
the upper bound of $\nu_{k}(M)/\nu_{1}(M)$ does not depend on $n$.
\end{rem}

%%%%%%%%%%%%%%%%%%%%%%%%
\subsection{Comparisons}\label{sec:Comparison theorems}
The author \cite{S1}, \cite{S2} has obtained inscribed radius comparison theorems,
and rigidity results for the equality case.
We first recall the following comparison (see Theorem 1.1 in \cite{S1}):
\begin{thm}[\cite{S1}]\label{thm:inscribed radius comparison}
For $N\in [n,\infty)$,
we assume $\ric^{N}_{f,\bm^{\perp}}\geq (N-1)\kappa$ and $H_{f,\bm}\geq (N-1)\lambda$.
If $\kappa$ and $\lambda$ satisfy the ball-condition,
then
\begin{equation*}
\IR M \leq \const.
\end{equation*}
\end{thm}
\begin{rem}
Kasue \cite{K1} has proved Theorem \ref{thm:inscribed radius comparison},
and rigidity result for the equality case in the unweighted case where $f=0$ and $N=n$.
Li and Wei have done in \cite{LW2} when $\kappa=0$,
and in \cite{LW1} when $\kappa<0$.
\end{rem}

We also have the following comparison (see Theorem 6.3 in \cite{S2}):
\begin{thm}[\cite{S2}]\label{thm:twisted inscribed radius comparison}
Assume
$\ric^{1}_{f,\bm^{\perp}}\geq (n-1) \,\kappa \,e^{\frac{-4f}{n-1}}$ and $H_{f,\bm}\geq (n-1)\,\lambda\,e^{\frac{-2f}{n-1}}$.
Suppose additionally that
$f\leq (n-1)\delta$ for some $\delta \in \mathbb{R}$.
If $\kappa$ and $\lambda$ satisfy the ball-condition,
then we have
\begin{equation*}
\IR M \leq C_{\kappa\,e^{-4\delta},\lambda\,e^{-2\delta}}.
\end{equation*}
\end{thm}

For $\kappa,\lambda \in \mathbb{R}$,
let $s_{\kappa,\lambda}(t)$ be a unique solution of the Jacobi equation $\psi''(t)+\kappa\, \psi(t)=0$ with $\psi(0)=1,\,\psi'(0)=-\lambda$.
Notice that
$\kappa$ and $\lambda$ satisfy the ball-condition if and only if
the equation $s_{\kappa,\lambda}(t)=0$ has a positive solution;
moreover,
$\const=\inf \{\,t>0 \mid s_{\kappa,\lambda}(t)=0\, \}$.
We define $\bar{C}_{\kappa,\lambda}$ as follows:
If $\kappa$ and $\lambda$ satisfy the ball-condition,
then $\bar{C}_{\kappa,\lambda}:=\const$;
otherwise,
$\bar{C}_{\kappa,\lambda}:=\infty$.
For $N\in (1,\infty)$,
let us define a function $s_{N,\kappa,\lambda}:(0,\infty]\to (0,\infty]$ by
\begin{equation}\label{eq:one dimensional volume growth}
\bar{s}_{\kappa,\lambda}(t) := \begin{cases}
                                            s_{\kappa,\lambda}(t) & \text{if $t< \bar{C}_{\kappa,\lambda}$},\\
                                            0                     & \text{if $t\geq    \bar{C}_{\kappa,\lambda}$},
                                 \end{cases}
                                 \quad s_{N,\kappa,\lambda}(r):=\int^{r}_{0}\, \bar{s}^{N-1}_{\kappa,\lambda}(t)\,dt.
\end{equation}

\begin{rem}\label{rem:formulation of finite dimensional model spaces}
For $\kappa,\lambda\in \mathbb{R}$,
the value $s_{n,\kappa,\lambda}(\bar{C}_{\kappa,\lambda})$ is finite if and only if
either (1) $\kappa$ and $\lambda$ satisfy the ball-condition;
or (2) $\kappa<0$ and $\lambda=\sqrt{\vert \kappa \vert}$ (cf. formulation of finite dimensional model spaces in Subsection \ref{sec:Comparisons and concentration phenomena}).
\end{rem}

For $r>0$ and $\Omega \subset M$,
let $B_{r}(\Omega)$ stand for the closed $r$-neighborhood of $\Omega$.
We next recall the following relative volume comparison theorem for metric neighborhoods of boundaries (see Theorem 5.4 in \cite{S1}):
\begin{thm}[\cite{S1}]\label{thm:volume comparison}
Let $\bm$ be compact.
For $N\in [n,\infty)$,
we assume $\ric^{N}_{f,\bm^{\perp}}\geq (N-1)\kappa$ and $H_{f,\bm}\geq (N-1)\lambda$.
Then for all $r,R>0$ with $r\leq R$
we have
\begin{equation*}
\frac{m_{M,f}(B_{R}(\bm))}{m_{M,f}(B_{r}(\bm))}\leq \frac{s_{N,\kappa,\lambda}(R)}{s_{N,\kappa,\lambda}(r)}.
\end{equation*}
\end{thm}
\begin{rem}
Under the same setting as in Theorem \ref{thm:volume comparison},
Bayle \cite{B} has stated a similar absolute volume comparison of Heintze-Karcher type (see Theorem E.2.2 in \cite{B}, and see also \cite{HK}, \cite{Mi3}, \cite{Mi4}, \cite{Mo}).
\end{rem}

We further recall the following comparison (see Theorem 7.6 in \cite{S2}):
\begin{thm}[\cite{S2}]\label{thm:twisted volume comparison}
Let $\bm$ be compact.
Let us assume $\ric^{1}_{f,\bm^{\perp}}\geq (n-1) \,\kappa \,e^{\frac{-4f}{n-1}}$ and $H_{f,\bm}\geq (n-1)\,\lambda\,e^{\frac{-2f}{n-1}}$.
Suppose additionally that
$f\leq (n-1)\delta$ for some $\delta \in \mathbb{R}$.
Assume that
one of the following holds:
\begin{enumerate}
\item $\kappa$ and $\lambda$ satisfy the convex-ball-condition;
\item $\kappa \leq 0$ and $\lambda=\sqrt{\vert \kappa \vert}$.
\end{enumerate}
Then for all $r,R>0$ with $r\leq R$
we have
\begin{equation*}
\frac{m_{M,f}(B_{R}(\bm))}{m_{M,f}(B_{r}(\bm))}\leq \frac{s_{n,\kappa\,e^{-4\delta},\lambda\,e^{-2\delta}}(R)}{s_{n,\kappa\,e^{-4\delta},\lambda\,e^{-2\delta}}(r)}.
\end{equation*}
\end{thm}

\begin{rem}\label{rem:remark on the curvature assumption}
In \cite{S1},
the author has stated that
the comparison inequalities in Theorems \ref{thm:inscribed radius comparison} and \ref{thm:volume comparison} hold
under the curvature condition $\ric^{N}_{f,M}\geq (N-1)\kappa$ and $H_{f,\bm}\geq (N-1)\lambda$.
Actually,
the author \cite{S1} has proved such comparison inequalities under the curvature condition $\ric^{N}_{f,\bm^{\perp}}\geq (N-1)\kappa$ and $H_{f,\bm}\geq (N-1)\lambda$ relying on the Laplacian comparison for the distance function $\rho_{\bm}$.
We can say the same thing for the inequalities in Theorems \ref{thm:twisted inscribed radius comparison} and \ref{thm:twisted volume comparison} (see \cite{S2}).
\end{rem}

\section{Invariants}\label{sec:Invariants}

%%%%%%%%%%%%%%%%%%%%%%%%
\subsection{Isomophisms}
We first introduce the following notion:
\begin{defi}\label{defi:dominance}
For $i=1,2$,
let $X_{i}=(X_{i},d_{X_{i}},\mu_{X_{i}})$ be metric measure spaces with boundary.
We say that \textit{$X_{1}$ dominates $X_{2}$} if
there exists a $1$-Lipschitz map $\Phi:X_{1}\to X_{2}$ such that
\begin{equation}\label{eq:dominance}
\Phi_{\#}\mu_{X_{1}}=\mu_{X_{2}},\quad \Phi(\partial X_{1})\subset \partial X_{2},
\end{equation}
where $\Phi_{\#}\mu_{X_{1}}$ denotes the push-forward of $\mu_{X_{1}}$ by $\Phi$.
We also say that
$X_{1}$ and $X_{2}$ are \textit{isomorphic} to each other if
they dominate each other.
\end{defi}

\begin{rem}\label{rem:mm-spaces}
A triple $\mathcal{X}=(\mathcal{X},r_{\mathcal{X}},\sigma_{\mathcal{X}})$ is said to be an \textit{mm-space} when
$(\mathcal{X},r_{\mathcal{X}})$ is a complete separable metric space,
and $\sigma_{\mathcal{X}}$ is a Borel probability measure on $\mathcal{X}$.
For $i=1,2$,
let $\mathcal{X}_{i}=(\mathcal{X}_{i},r_{\mathcal{X}_{i}},\sigma_{\mathcal{X}_{i}})$ be mm-spaces.
They are said to be \textit{mm-isomorphic} to each other if 
there exists an isometry $\Phi:\supp \sigma_{\mathcal{X}_{1}}\to \supp \sigma_{\mathcal{X}_{2}}$ such that $\Phi_{\#}\sigma_{\mathcal{X}_{1}}=\sigma_{\mathcal{X}_{2}}$,
where $\supp \sigma_{\mathcal{X}_{i}}$ are the support of $\sigma_{\mathcal{X}_{i}}$.
It is said that
$\mathcal{X}_{1}$ \textit{dominates} $\mathcal{X}_{2}$ if
there exists a $1$-Lipschitz map $\Psi:\mathcal{X}_{1}\to \mathcal{X}_{2}$ such that $\Psi_{\#}\sigma_{\mathcal{X}_{1}}=\sigma_{\mathcal{X}_{2}}$.
It is well-known that
if $\mathcal{X}_{1}$ and $\mathcal{X}_{2}$ dominate each other,
then they are mm-isomorphic to each other (see e.g., Proposition 2.11 in \cite{Sh}).

For $i=1,2$,
let $X_{i}=(X_{i},d_{X_{i}},\mu_{X_{i}})$ be two metric measure spaces with boundary.
If $X_{1}$ and $X_{2}$ are isomorphic to each other in the sense of Definition \ref{defi:dominance},
then they dominate each other as mm-spaces;
in particular,
they are mm-isomorphic to each other.
\end{rem}

Let us show the following monotonicity of the partial inscribed radius defined as (\ref{eq:partial inscribed radius}) (cf. Proposition 2.18 in \cite{Sh}):
\begin{lem}\label{lem:monotonicity of the partial inscribed radius}
For $i=1,2$,
let $X_{i}=(X_{i},d_{X_{i}},\mu_{X_{i}})$ be metric measure spaces with boundary.
If $X_{1}$ dominates $X_{2}$,
then for every $\eta>0$
\begin{equation*}
\PI(X_{2};1-\eta)\leq \PI(X_{1};1-\eta).
\end{equation*}
\end{lem}
\begin{proof}
By the assumption for the dominance,
there exists a $1$-Lipschitz map $\Phi:X_{1}\to X_{2}$ such that (\ref{eq:dominance}).
We fix a Borel subset $\Omega \subset X_{1}$ with $\mu_{X_{1}}(\Omega)\geq 1-\eta$.
From $\Phi_{\#}\mu_{X_{1}}=\mu_{X_{2}}$
it follows that
\begin{equation}\label{lem:first key of the monotonicity of the partial inscribed radius}
\mu_{X_{2}}\bigl( \overline{\Phi(\Omega)}  \bigl)=\mu_{X_{1}}\bigl(\Phi^{-1}(\overline{\Phi(\Omega)} )\bigl) \geq \mu_{X_{1}}(\Omega)\geq 1-\eta,
\end{equation}
where $\overline{\Phi(\Omega)}$ denotes the closure of $\Phi(\Omega)$.

We now show
\begin{equation}\label{lem:second key of the monotonicity of the partial inscribed radius}
\IR \overline{\Phi(\Omega)} \leq \IR \Omega,
\end{equation}
where $\IR \overline{\Phi(\Omega)}$ and $\IR \Omega$ are the inscribed radii of $\overline{\Phi(\Omega)}$ and $\Omega$ defined as (\ref{eq:inscribed radius}),
respectively.
We fix $x_{2}\in \overline{\Phi(\Omega)}$,
and take a sequence $\{ x_{2,j} \}$ in $\Phi(\Omega)$ with $x_{2,j}\to x_{2}$.
For each $j$
we have $x_{1,j}\in \Omega$ satisfying $x_{2,j}=\Phi(x_{1,j})$.
By using the properness of $X_{1}$,
we can choose $z_{1,j}\in \partial X_{1}$ with $d_{X_{1}}(x_{1,j},z_{1,j})=\rho_{\partial X_{1}}(x_{1,j})$.
We set $z_{2,j}:=\Phi(x_{2,j})$.
From $\Phi(\partial X_{1})\subset \partial X_{2}$
we deduce $z_{2} \in \partial X_{2}$.
Since $\Phi$ is $1$-Lipschitz,
we see
\begin{equation*}
\rho_{\partial X_{2}}(x_{2,j})\leq d_{X_{2}}(x_{2,j},z_{2,j})\leq d_{X_{1}}(x_{1,j},z_{1,j})=\rho_{\partial X_{1}}(x_{1,j})\leq \IR \Omega.
\end{equation*}
Letting $j\to \infty$,
we derive $\rho_{\partial X_{2}}(x_{2})\leq \IR \Omega$.
This yields (\ref{lem:second key of the monotonicity of the partial inscribed radius}).
In virtue of (\ref{lem:first key of the monotonicity of the partial inscribed radius}) and (\ref{lem:second key of the monotonicity of the partial inscribed radius}),
we obtain
\begin{equation*}
\PI(X_{2};1-\eta)\leq \IR \overline{\Phi(\Omega)} \leq \IR \Omega.
\end{equation*}
The arbitrariness of $\Omega$ completes the proof.
\end{proof}

We also have the following monotonicity of the observable inscribed radius
introduced in Definition \ref{defi:observable inscribed radius} (cf. Proposition 2.18 in \cite{Sh}).
The proof is straightforward,
and we omit it.
\begin{lem}\label{lem:monotonicity of the observable inscribed radius}
For $i=1,2$,
let $X_{i}=(X_{i},d_{X_{i}},\mu_{X_{i}})$ be metric measure spaces with boundary.
If $X_{1}$ dominates $X_{2}$,
then for every $\eta>0$
\begin{equation*}
\OI(X_{2};-\eta)\leq \OI(X_{1};-\eta).
\end{equation*}
\end{lem}
%\begin{proof}
%There is a $1$-Lipschitz map $\Phi:X_{1}\to X_{2}$ such that (\ref{eq:dominance}).
%We fix a $1$-Lipschitz function $\varphi_{2}:X_{2}\to I$ with $\varphi_{2}|_{\partial X_{2}}=0$,
%and define a $1$-Lipschitz function $\varphi_{1}:X_{1}\to I$ by $\varphi_{1}:=\varphi_{2} \circ \Phi$.
%Notice that
%$\varphi_{1}|_{\partial X_{1}}=0$ and $m_{I,\varphi_{1}}=m_{I,\varphi_{2}}$,
%where $m_{I,\varphi_{i}}$ are the Borel probability measures on the $\varphi_{i}$-screens $I_{\varphi_{i}}$ defined as (\ref{eq:screen}).
%Hence
%we have
%\begin{equation*}
%\PI\left(I_{\varphi_{2}};1-\eta \right)=\PI \bigl(I_{\varphi_{1}};1-\eta \bigl)\leq \OI(X_{1};-\eta).
%\end{equation*}
%This implies the desired monotonicity.
%\end{proof}
According to Lemmas \ref{lem:monotonicity of the partial inscribed radius} and \ref{lem:monotonicity of the observable inscribed radius},
they are invariants under the isomorphism introduced in Definition \ref{defi:dominance}.

%%%%%%%%%%%%%%%%%%%%%%%%
\subsection{Boundary separation distances}
Let $X=(X,d_{X},\mu_{X})$ be a metric measure space with boundary,
and let $k\geq 1$ be an integer.
For positive numbers $\eta_{1},\dots,\eta_{k}>0$,
we denote by $\mathcal{S}_{X}(\eta_{1},\dots,\eta_{k})$ the set of all sequences $\{\Omega_{\alpha}\}^{k}_{\alpha=1}$ of Borel subsets $\Omega_{\alpha}$ with $\mu_{X}(\Omega_{\alpha})\geq \eta_{\alpha}$.
For a sequence $\{\Omega_{\alpha}\}^{k}_{\alpha=1} \in \mathcal{S}_{X}(\eta_{1},\dots,\eta_{k})$,
we set
\begin{equation*}
\mathcal{D}_{X}\bigl(\{\Omega_{\alpha}\}^{k}_{\alpha=1}\bigl):=\min \left \{\, \min_{\alpha \neq \beta} d_{X}(\Omega_{\alpha},\Omega_{\beta}),\,\, \min_{\alpha} d_{X}(\Omega_{\alpha},\partial X)  \,\right\}.
\end{equation*}

We now define the following quantity:
\begin{defi}\label{defi:Dirichlet separation distance}
Let $X=(X,d_{X},\mu_{X})$ be a metric measure space with boundary.
For $\eta_{1},\dots,\eta_{k}>0$,
we define the \textit{$(\eta_{1},\dots,\eta_{k})$-boundary separation distance $\DS(X;\eta_{1},\dots,\eta_{k})$ of $X$} as follows:
If $\mathcal{S}_{X}(\eta_{1},\dots,\eta_{k})$ is non-empty,
then 
\begin{equation*}
\DS(X;\eta_{1},\dots,\eta_{k}):=\sup\,\mathcal{D}_{X}\bigl(\{\Omega_{\alpha}\}^{k}_{\alpha=1}\bigl),
\end{equation*}
where the supremum is taken over all $\{\Omega_{\alpha}\}^{k}_{\alpha=1} \in \mathcal{S}_{X}(\eta_{1},\dots,\eta_{k})$;
otherwise,
$\DS(X;\eta_{1},\dots,\eta_{k}):=0$.
\end{defi}
The boundary separation distance $\DS(X;\eta_{1},\dots,\eta_{k})$ is monotone non-increasing in $\eta_{\alpha}$ for each $\alpha=1,\dots,k$.

\begin{rem}\label{rem:Gromov separation distance}
The boundary separation distance is an analogue of the \textit{separation distance} on mm-spaces introduced by Gromov \cite{G}.
For later convenience,
we recall its precise definition:
Let $\mathcal{X}=(\mathcal{X},r_{\mathcal{X}},\sigma_{\mathcal{X}})$ be an mm-space (see Remark \ref{rem:mm-spaces}).
For positive numbers $\eta_{0}, \eta_{1},\dots,\eta_{k}>0$,
the \textit{$(\eta_{0}, \eta_{1},\dots,\eta_{k})$-separation distance} is defined as
\begin{equation}\label{eq:Gromov separation distance}
\SE(\mathcal{X};\eta_{0}, \eta_{1},\dots,\eta_{k}):=\sup\, \min_{\alpha \neq \beta} r_{\mathcal{X}}(\Omega_{\alpha},\Omega_{\beta}),
\end{equation}
where the supremum is taken over all sequences $\{\Omega_{\alpha}\}^{k}_{\alpha=0}$ of Borel subsets $\Omega_{\alpha} \subset \mathcal{X}$ with $\sigma_{\mathcal{X}}(\Omega_{\alpha})\geq \eta_{\alpha}$.
If there exists no such sequence,
then we set $\SE(\mathcal{X};\eta_{0}, \eta_{1},\dots,\eta_{k}):=0$.
\end{rem}

We verify the following monotonicity (cf. Lemma 2.25 in \cite{Sh}):
\begin{lem}\label{lem:monotonicity of Dirichlet separation distance}
For $i=1,2$,
let $X_{i}=(X_{i},d_{X_{i}},\mu_{X_{i}})$ be metric measure spaces with boundary.
If $X_{1}$ dominates $X_{2}$,
then for all $\eta_{1},\dots,\eta_{k}>0$
\begin{equation*}
\DS(X_{2};\eta_{1},\dots,\eta_{k})\leq \DS(X_{1};\eta_{1},\dots,\eta_{k}).
\end{equation*}
\end{lem}
\begin{proof}
We may assume $\mathcal{S}_{X_{2}}(\eta_{1},\dots,\eta_{k})\neq \emptyset$.
Fix a sequence $\{\Omega_{\alpha}\}^{k}_{\alpha=1} \in \mathcal{S}_{X_{2}}(\eta_{1},\dots,\eta_{k})$.
There exists a $1$-Lipschitz map $\Phi:X_{1}\to X_{2}$ such that (\ref{eq:dominance}).
Since $\Phi$ is $1$-Lipschitz,
for all $\alpha,\beta=1,\dots,k$,
we see
\begin{equation}\label{eq:monotonicity of distances between Borel subsets}
d_{X_{2}}(\Omega_{\alpha},\Omega_{\beta})\leq d_{X_{1}}(\Phi^{-1}(\Omega_{\alpha}),\Phi^{-1}(\Omega_{\beta})).
\end{equation}
Furthermore,
by $\Phi_{\#}\mu_{X_{1}}=\mu_{X_{2}}$,
for every $\alpha=1,\dots,k$,
\begin{equation}\label{eq:lower bound of volume}
\mu_{X_{1}}(\Phi^{-1}(\Omega_{\alpha}))=\mu_{X_{2}}(\Omega_{\alpha}) \geq \eta_{\alpha}.
\end{equation}

We show that
for every $\alpha=1,\dots,k$,
\begin{equation}\label{eq:distance monotonicity}
d_{X_{2}}(\Omega_{\alpha},\partial X_{2})\leq d_{X_{1}}(\Phi^{-1}(\Omega_{\alpha}),\partial X_{1}).
\end{equation}
Take $x_{1} \in \Phi^{-1}(\Omega_{\alpha})$,
and put $x_{2}:=\Phi(x_{1})\in \Omega_{\alpha}$.
From the properness of $X_{1}$,
there exists $z_{1} \in \partial X_{1}$ such that $d_{X_{1}}(x_{1},z_{1})=d_{X_{1}}(x_{1},\partial X_{1})$.
Put $z_{2}:=\Phi(z_{1})\in \partial X_{2}$.
It follows that
\begin{equation*}
d_{X_{2}}(\Omega_{\alpha},\partial X_{2})\leq d_{X_{2}}(x_{2},z_{2})\leq d_{X_{1}}(x_{1},z_{1})=d_{X_{1}}(x_{1},\partial X_{1}),
\end{equation*}
and
we obtain (\ref{eq:distance monotonicity}).
By combining (\ref{eq:monotonicity of distances between Borel subsets}), (\ref{eq:lower bound of volume}) and (\ref{eq:distance monotonicity}),
\begin{equation*}
\mathcal{D}_{X_{2}}\bigl(\{\Omega_{\alpha}\}^{k}_{\alpha=1}\bigl) \leq \mathcal{D}_{X_{1}}\bigl(\{   \Phi^{-1}(\Omega_{\alpha})   \}^{k}_{\alpha=1}\bigl) \leq \DS(X_{1};\eta_{1},\dots,\eta_{k}).
\end{equation*}
This completes the proof.
\end{proof}
Lemma \ref{lem:monotonicity of Dirichlet separation distance} tells us that
the boundary separation distance is an invariant under the isomorphism introduced in Definition \ref{defi:dominance}.

%%%%%%%%%%%%%%%%%%%%%%%%
\subsection{Relations between invariants}

We present the following relation between our invariants (cf. Proposition 2.26 in \cite{Sh}):
\begin{lem}\label{lem:observable inscribed radius is smaller than Dirichlet separation distance}
Let $X=(X,d_{X},\mu_{X})$ be a metric measure space with boundary.
Then for every $\eta>0$
we have
\begin{equation*}
\OI(X;-\eta)\leq \DS(X;\eta).
\end{equation*}
In particular,
$\OI(X;-\eta)\leq \IR X$.
\end{lem}
\begin{proof}
We may assume $\eta<1$.
We fix a $1$-Lipschitz function $\varphi:X\to I$ with $\varphi|_{\partial X}=0$.
Let $m_{I,\varphi}$ be the Borel probability measure on the $\varphi$-screen $I_{\varphi}$ defined as (\ref{eq:screen}).
By $m_{I,\varphi}=\varphi_{\#}\mu_{X}$ and $\varphi(\partial X)\subset \partial I$,
the space $X$ dominates $I_{\varphi}$;
in particular,
Lemma \ref{lem:monotonicity of Dirichlet separation distance} yields
\begin{equation}\label{eq:monotonicity of Dirichlet separation distance under the screen}
\DS(I_{\varphi};\eta) \leq \DS(X;\eta).
\end{equation}
We put
\begin{equation*}
t_{0}:=\inf\{\,t\in I \mid m_{I,\varphi}((t,\infty))\leq \eta \,\}.
\end{equation*}
Note that
$m_{I,\varphi}((t_{0},\infty))\leq \eta$ and $m_{I,\varphi}([t_{0},\infty))\geq \eta$.
Since $m_{I,\varphi}([0,t_{0}])\geq 1-\eta$,
we have
\begin{equation}\label{eq:upper partial inscribed radius bound}
\PI(I_{\varphi};1-\eta)\leq \IR [0,t_{0}]= t_{0}.
\end{equation}
On the other hand,
$m_{I,\varphi}([t_{0},\infty))\geq \eta$ leads to
\begin{equation}\label{eq:lower Dirichlet separation distance bound}
t_{0}=d_{I}(\partial I,[t_{0},\infty)) \leq \DS(I_{\varphi};\eta).
\end{equation}
Now,
(\ref{eq:monotonicity of Dirichlet separation distance under the screen}) together with (\ref{eq:upper partial inscribed radius bound}), (\ref{eq:lower Dirichlet separation distance bound}) implies the desired one.
\end{proof}

\begin{rem}
A sequence $\{X_{n}\}$ of metric measure spaces with boundary is said to be \textit{inscribed radius collapsing} if 
$\IR X_{n} \to 0$ as $n\to \infty$.
Yamaguchi and Zhang \cite{YZ} have studied inscribed radius collapsing sequences of manifolds with boundary from the view point of the collapsing theory.
From Lemma \ref{lem:observable inscribed radius is smaller than Dirichlet separation distance},
it follows that
if a sequence of metric measure spaces with boundary is inscribed radius collapsing,
then it is a boundary concentration family introduced in Definition \ref{defi:Dirichlet Levy family}.
\end{rem}

We also possess the following relation (cf. Proposition 2.26 in \cite{Sh}):
\begin{lem}\label{lem:Dirichlet separation distance is smaller than observable inscribed radius}
Let $X=(X,d_{X},\mu_{X})$ be a metric measure space with boundary.
Then for all $\eta,\eta'>0$ with $\eta>\eta'$
we have
\begin{equation*}
\DS(X;\eta) \leq \OI(X;-\eta').
\end{equation*}
\end{lem}
\begin{proof}
We may assume $\DS(X;\eta)>0$.
Fix a Borel subset $\Omega \subset X$ with $\mu_{X}(\Omega)\geq \eta$,
and also fix $J \subset I$ with $m_{I,\rho_{\partial X}}(J) \geq 1-\eta'$.
Then
\begin{equation*}
m_{I,\rho_{\partial X}}(\overline{\rho_{\partial X}(\Omega)})+m_{I,\rho_{\partial X}}(J)\geq \mu_{X}(\Omega)+m_{I,\rho_{\partial X}}(J)\geq \eta+(1-\eta') >1,
\end{equation*}
and hence $\overline{\rho_{\partial X}(\Omega)} \cap J \neq \emptyset$,
where $\overline{\rho_{\partial X}(\Omega)}$ denotes the closure of $\rho_{\partial X}(\Omega)$.
For every $t_{0} \in \overline{\rho_{\partial X}(\Omega)} \cap J$
we see
\begin{align*}
\IR J &\geq d_{I}(t_{0},\partial I) \geq d_{I}(\overline{\rho_{\partial X}(\Omega)},\partial I)\\
        &=d_{I}(\rho_{\partial X}(\Omega),\partial I)=\inf_{x\in \Omega} \rho_{\partial X}(x)=d_{X}(\Omega,\partial X).
\end{align*}
This yields
\begin{equation*}
\OI(X;-\eta') \geq \PI(I_{\rho_{\partial X}};1-\eta')\geq \DS(X;\eta).
\end{equation*}
We arrive at the desired assertion.
\end{proof}

By combining Lemmas \ref{lem:observable inscribed radius is smaller than Dirichlet separation distance} and \ref{lem:Dirichlet separation distance is smaller than observable inscribed radius},
and by straightforward argument,
one can conclude the following equivalence:
\begin{prop}\label{prop:characterization of Dirichlet Levy}
Let $\{X_{n}\}$ be a sequence of metric measure spaces with boundary $X_{n}=(X_{n},d_{X_{n}},\mu_{X_{n}})$.
Then the following are equivalent:
\begin{enumerate}\setlength{\itemsep}{+0.7mm}
\item $\{X_{n}\}$ is a boundary concentration family;
\item for every sequence $\{\Omega_{n}\}$ of Borel subsets $\Omega_{n} \subset X_{n}$ satisfying $\liminf_{n\to \infty} \mu_{X_{n}}(\Omega_{n})>0$,
         we have $\lim_{n\to \infty}\, d_{X_{n}}(\Omega_{n},\partial X_{n})=0$;
\item for every $r>0$
         we have $\lim_{n\to \infty}\,\mu_{X_{n}}(B_{r}(\partial X_{n}))=1$.
\end{enumerate}
\end{prop}

Finally,
we observe the following fact for the equality case of Lemma \ref{lem:Dirichlet separation distance is smaller than observable inscribed radius}.
The statement and its proof are informed by Daisuke Kazukawa.
\begin{prop}\label{prop:Kazukawa}
Let $X=(X,d_{X},\mu_{X})$ be a metric measure space with boundary.
If $\supp \mu_{X}=X$,
then for every $\eta>0$
we have
\begin{equation*}
\OI(X;-\eta)=\DS(X;\eta).
\end{equation*}
\end{prop}
\begin{proof}
Lemma \ref{lem:observable inscribed radius is smaller than Dirichlet separation distance} tells us that
the left hand side is at most the right hand side.
We verify the opposite.
We may assume $\DS(X;\eta)>0$.
Fix a sufficiently small $\epsilon>0$.
Then there exists a Borel subset $\Omega \subset X$ with $\mu_{X}(\Omega) \geq \eta$ such that $d_{X}(\Omega,\partial X) >\DS(X;\eta)-\epsilon$.
Let us define a $1$-Lipschitz function $\varphi:X\to I$ by
\begin{equation*}
\varphi(x):=\max \left\{d_{X}(\Omega,\partial X)-d_{X}(x,\Omega),0\right\}.
\end{equation*}
Notice that
$\varphi|_{\partial X}=0,\,\varphi|_{\Omega}=d_{X}(\Omega,\partial X)$ and
\begin{equation*}
m_{I,\varphi}(\{ d_{X}(\Omega,\partial X)\})\geq \mu_{X}(\Omega) \geq \eta.
\end{equation*}
Furthermore,
$\supp \mu_{X}=X$ implies $\supp m_{I,\varphi}=[0,d_{X}(\Omega,\partial X)]$.

We now show
\begin{equation}\label{eq:Kazukawa}
\PI(I_{\varphi};1-\eta) \geq d_{X}(\Omega,X).
\end{equation}
The proof is by contradiction.
Suppose that
there exists a Borel subset $J \subset I$ with $m_{I,\varphi}(J)\geq 1-\eta$ such that $\IR J<d_{X}(\Omega,\partial X)$.
We put $\hat{J}:=(\IR J, d_{X}(\Omega,\partial X))$.
Then we have
\begin{align*}
1&   =   m_{I,\varphi}([0,\IR J])+m_{I,\varphi}(\hat{J})+m_{I,\varphi}(\{ d_{X}(\Omega,\partial X)\})\\
  &\geq m_{I,\varphi}(J)+m_{I,\varphi}(\hat{J})+\eta \geq 1+m_{I,\varphi}(\hat{J}),
\end{align*}
and hence $m_{I,\varphi}(\hat{J})=0$.
This contradicts $\supp m_{I,\varphi}=[0,d_{X}(\Omega,\partial X)]$.
From (\ref{eq:Kazukawa})
it follows that
\begin{equation*}
\DS(X;\eta)-\epsilon<d_{X}(\Omega,\partial X) \leq \OI(X;-\eta).
\end{equation*}
By letting $\epsilon \to 0$,
we complete the proof.
\end{proof}
\section{Dirichlet eigenvalues}\label{sec:Dirichlet eigenvalues}
In what follows,
let $(M,d_{M},m_{M,f})$ denote an $n$-dimensional weighted Riemannian manifold with boundary defined as (\ref{eq:weighted volume measure}) such that $m_{M,f}$ is a Borel probability measure.
We study the metric measure space with boundary $(M,f)$ defined as (\ref{eq:weighted Riemannian manifold}).

%%%%%%%%%%%%%%%%%%%%%%%%
\subsection{Boundary separation distances and Dirichlet eigenvalues}
Let us show the following relation between the boundary separation distance and the Dirichlet eigenvalue:
\begin{lem}\label{lem:Dirichlet eigenvalue and Dirichlet separation distance}
Let $M$ be compact.
Then for all $\eta_{1},\dots,\eta_{k}>0$
we have
\begin{equation*}
\DS\left((M,f);\eta_{1},\dots,\eta_{k}\right)\leq \frac{2}{\sqrt{\nu_{f,k}(M)\, \min_{\alpha=1,\dots,k} \eta_{\alpha}}},
\end{equation*}
where $\nu_{f,k}(M)$ is the $k$-th Dirichlet eigenvalue of $\Delta_{f}$ defined as $(\ref{eq:notation of Dirichlet eigenvalue})$.
\end{lem}
\begin{proof}
We set $S:=\DS((M,f);\eta_{1},\dots,\eta_{k})$.
We may assume $S>0$.
Let us fix a sufficiently small $\epsilon>0$.
There exists a sequence $\{\Omega_{\alpha}\}^{k}_{\alpha=1} \in \mathcal{S}_{(M,f)}(\eta_{1},\dots,\eta_{k})$ such that $S-\epsilon<\mathcal{D}_{(M,f)}\bigl(\{\Omega_{\alpha}\}^{k}_{\alpha=1}\bigl)$.
Put $S_{\epsilon}:=S-\epsilon$.
For each $\alpha$,
we define a Lipschitz function $\phi_{\alpha}:M\to \mathbb{R}$ by
\begin{equation*}
\phi_{\alpha}(x):=\max \left\{1-\frac{2}{S_{\epsilon}}\,d_{M}(x,\Omega_{\alpha}),0\right\}.
\end{equation*}
Notice that
the support of $\phi_{\alpha}$ coincides with $B_{S_{\epsilon}/2}(\Omega_{\alpha})$.
Furthermore,
the following properties hold:
\begin{enumerate}\setlength{\itemsep}{+1.5mm}
\item $\phi_{\alpha} \equiv 0$ on $B_{S_{\epsilon}/2}(\bm)$; \label{item:vanish on boundary}
\item $\phi_{\alpha} \equiv 0$ on $B_{S_{\epsilon}/2}(\Omega_{\beta})$ for every $\beta=1,\dots,k$ with $\beta \neq \alpha$; \label{item:vanish on another domain}
\item $\phi_{\alpha} \equiv 1$ on $\Omega_{\alpha}$; \label{item:identity}
\item $\Vert \nabla \phi_{\alpha} \Vert \leq 2/S_{\epsilon}$ $m_{M,f}$-almost everywhere on $M$. \label{item:gradient estimate}
\end{enumerate}
By (\ref{item:vanish on boundary}),
$\phi_{\alpha}$ belongs to the Sobolev space $H^{1}_{0}(M,m_{M,f})$.
By (\ref{item:vanish on another domain}),
the functions $\phi_{1},\dots,\phi_{k}$ are orthogonal to each other in $H^{1}_{0}(M,m_{M,f})$.
Let $L_{0}$ be the $k$-dimensional subspace of $H^{1}_{0}(M,m_{M,f})$ spanned by $\phi_{1},\dots,\phi_{k}$.
From (\ref{item:identity}) and (\ref{item:gradient estimate}),
for every $\phi \in L_{0}\setminus \{0\}$
we deduce
\begin{align*}
\int_{M}\,\phi^{2}\,dm_{M,f}&=\sum^{k}_{\alpha=1}\,c^{2}_{\alpha}\,\int_{M}\,\phi^{2}_{\alpha}\,dm_{M,f} \geq \sum^{k}_{\alpha=1}\,c^{2}_{\alpha}\, \eta_{\alpha} \geq \min_{\alpha} \eta_{\alpha}\,\sum^{k}_{\alpha=1}\,c^{2}_{\alpha} ,\\
\int_{M}\,\Vert \nabla \phi \Vert^{2}\,dm_{M,f}&=\sum^{k}_{\alpha=1}\,c^{2}_{\alpha}\,\int_{M}\, \Vert \nabla \phi_{\alpha} \Vert^{2}\,dm_{M,f} \leq \left(\frac{2}{S_{\epsilon}}\right)^{2} \sum^{k}_{\alpha=1}\,c^{2}_{\alpha},
\end{align*}
where $c_{1},\dots,c_{k}$ are determined by $\phi=\sum^{k}_{\alpha=1}\,c_{\alpha} \phi_{\alpha}$.
Hence
we have
\begin{equation*}
R_{f}(\phi)=\frac{\int_{M} \,\Vert \nabla \phi \Vert^{2} \,dm_{M,f}}{\int_{M} \,\phi^{2} \,dm_{M,f}} \leq \frac{1}{\min_{\alpha} \eta_{\alpha}}\left(\frac{2}{S_{\epsilon}}\right)^{2},
\end{equation*}
where $R_{f}(\phi)$ is the Rayleigh quotient of $\phi$ defined as (\ref{eq:Rayleigh quotient}).
From the min-max principle (\ref{eq:min-max principle})
we derive
\begin{equation*}
\nu_{f,k}(M) \leq \sup_{\phi \in L_{0}\setminus \{0\}}\,R_{f}(\phi) \leq \frac{1}{\min_{\alpha} \eta_{\alpha}}\left(\frac{2}{S_{\epsilon}}\right)^{2}.
\end{equation*}
Letting $\epsilon \to 0$,
we conclude the inequality.
\end{proof}

\begin{rem}
In the forthcoming paper \cite{FSa},
Funano and the author prove the following refined estimate for $k=1$ (see Theorem 2.3 in \cite{FSa}):
For every $\eta>0$
we have
\begin{equation*}
\DS\left((M,f);\eta\right)\leq \frac{1}{\sqrt{\nu_{f,1}(M)}} \,\log \frac{e}{\eta}.
\end{equation*}
\end{rem}

\begin{rem}
Colbois and Savo \cite{CS} have shown a similar estimate for the $k$-th closed eigenvalue of the Laplacian (see Lemma 5 in \cite{CS}). 
\end{rem}

\begin{rem}
Chung, Grigor'yan and Yau \cite{CGY} have estimated the $k$-th closed eigenvalue and Robin eigenvalue of the Laplacian in terms of the separation distance (see Theorem 1.1 in \cite{CGY}).
By applying the same argument in \cite{CGY} to our setting,
we see the following estimate for $\nu_{f,k}(M)$:
Suppose that $M$ is compact.
Then for every integer $k>2$,
and for every sequence $\{\Omega_{\alpha}\}^{k}_{\alpha=1}$ of Borel subsets with $\min_{\alpha \neq \beta} d_{M}(\Omega_{\alpha},\Omega_{\beta})\geq D$,
\begin{equation*}
\nu_{f,k}(M)-\nu_{f,1}(M) \leq \frac{1}{D^{2}}\,\max_{\alpha \neq \beta} \,\left( \log\,\frac{4}{\int_{\Omega_{\alpha}} \,\phi^{2}_{f,1}\,dm_{M,f} \, \int_{\Omega_{\beta}} \,\phi^{2}_{f,1}\,dm_{M,f}   }  \right)^{2},
\end{equation*}
where $\phi_{f,1}$ is an $L^{2}$-normalized eigenfunction for $\nu_{f,1}(M)$.
\end{rem}

%%%%%%%%%%%%%%%%%%%%%%%%
\subsection{Observable inscribed radii and Dirichlet eigenvalues}
Now,
Lemma \ref{lem:Dirichlet eigenvalue and Dirichlet separation distance} together with Lemma \ref{lem:observable inscribed radius is smaller than Dirichlet separation distance}
leads us to the following relation between the observable inscribed radius and the Dirichlet eigenvalue:
\begin{prop}\label{prop:logarithmic Dirichlet eigenvalue and observable inscribed radius}
Let $M$ be compact.
Then for every $\eta>0$,
\begin{equation*}
\OI((M,f);-\eta)\leq \frac{2}{\sqrt{\nu_{f,1}(M)\, \eta}}.
\end{equation*}
\end{prop}

\begin{rem}
Under the curvature condition (\ref{eq:finite dimensional curvature bound}) and $\IR M \leq D$,
the author \cite{S1} has provided a lower bound of $\nu_{f,1}(M)$ depending only on $\kappa,\,\lambda,\,N$ and $D$ (see Theorem 1.6 in \cite{S1} and Remark \ref{rem:remark on the curvature assumption},
and also pioneering works of Li-Yau \cite{LY} and Kasue \cite{K2}).
Combining Proposition \ref{prop:logarithmic Dirichlet eigenvalue and observable inscribed radius} with the lower estimate tells us that
we have an upper bound of $\OI((M,f);-\eta)$ depending only on $\kappa,\,\lambda,\,N,\,D$ and $\eta$.

Similarly,
under the condition (\ref{eq:one dimensional curvature bound}) for $\kappa$ and $\lambda$ satisfying the convex-ball-condition,
and under $f\leq (n-1)\delta$,
the author \cite{S2} has obtained a lower bound of $\nu_{f,1}(M)$ depending only on $n,\,\kappa,\,\lambda$ and $\delta$ (see Theorem 8.5 in \cite{S2}).
In virtue of Proposition \ref{prop:logarithmic Dirichlet eigenvalue and observable inscribed radius},
we possess an upper bound of $\OI((M,f);-\eta)$ depending only on $n,\,\kappa,\,\lambda,\,\delta$ and $\eta$.
\end{rem}

Proposition \ref{prop:logarithmic Dirichlet eigenvalue and observable inscribed radius} enables us to prove Theorem \ref{thm:Dirichlet eigenvalue and Dirichlet Levy family}.
\begin{proof}[Proof of Theorem \ref{thm:Dirichlet eigenvalue and Dirichlet Levy family}]
Let $\{(M_{n},f_{n})\}$ be a sequence of compact metric measure spaces with boundary defined as $(\ref{eq:weighted Riemannian manifold})$.
By Proposition \ref{prop:logarithmic Dirichlet eigenvalue and observable inscribed radius},
if $\nu_{f_{n},1}(M_{n})\to \infty$ as $n\to \infty$,
then $\OI((M_{n},f_{n});-\eta)\to 0$.
Hence,
$\{(M_{n},f_{n})\}$ is a boundary concentration family.
\end{proof}

We also derive the following from Proposition \ref{prop:logarithmic Dirichlet eigenvalue and observable inscribed radius} and Theorem \ref{thm:universal inequality}:
\begin{thm}\label{thm:higer order Dirichlet eigenvalue and observable inscribed radius}
Let $M$ be compact.
If $\ric^{\infty}_{f,M}\geq 0$ and $H_{f,\bm}\geq 0$,
then there is a universal constant $C>0$ such that
for all $k\geq 1$ and $\eta>0$,
\begin{equation*}
\OI((M,f);-\eta)\leq \frac{C\,k}{\sqrt{\nu_{f,k}(M)\, \eta}}.
\end{equation*}
\end{thm}

Let us give a proof of Theorem \ref{thm:higher order Dirichlet eigenvalue and Dirichlet Levy family} by using Theorem \ref{thm:higer order Dirichlet eigenvalue and observable inscribed radius}.
\begin{proof}[Proof of Theorem \ref{thm:higher order Dirichlet eigenvalue and Dirichlet Levy family}]
Let $\{(M_{n},f_{n})\}$ denote a sequence of compact metric measure spaces with boundary defined as $(\ref{eq:weighted Riemannian manifold})$.
Assume that
$\ric^{\infty}_{f_{n},M_{n}}\geq 0$ and $H_{f_{n},\bm_{n}}\geq 0$.
Due to Theorem \ref{thm:higer order Dirichlet eigenvalue and observable inscribed radius},
if $\nu_{f_{n},k}(M_{n})\to \infty$ for $k$,
then $\OI((M_{n},f_{n});-\eta)\to 0$.
We complete the proof.
\end{proof}

\begin{rem}\label{rem:conjecture}
The author wonders whether the following statement of E. Milman type holds (see \cite{Mi1}, \cite{Mi2}, and Theorem 9.46 in \cite{Sh}):
Under the same setting as in Theorems \ref{thm:Dirichlet eigenvalue and Dirichlet Levy family}, \ref{thm:higher order Dirichlet eigenvalue and Dirichlet Levy family},
if $\ric^{\infty}_{f_{n},M_{n}}\geq 0$ and $H_{f_{n},\bm_{n}}\geq 0$,
and if $\{(M_{n},f_{n})\}$ is a boundary concentration family,
then $\nu_{f_{n},1}(M_{n})\to \infty$.
If it is true,
then
in virtue of Theorems \ref{thm:Dirichlet eigenvalue and Dirichlet Levy family}, \ref{thm:higher order Dirichlet eigenvalue and Dirichlet Levy family},
we can say that
under $\ric^{\infty}_{f_{n},M_{n}}\geq 0$ and $H_{f_{n},\bm_{n}}\geq 0$,
the following are equivalent:
\begin{enumerate}\setlength{\itemsep}{+0.7mm}
\item $\{(M_{n},f_{n})\}$ is a boundary concentration family;
\item $\nu_{f_{n},1}(M_{n})\to \infty$ as $n\to \infty$;
\item $\nu_{f_{n},k}(M_{n})\to \infty$ as $n\to \infty$ for some $k\geq 1$.
\end{enumerate}

The author also wonders if
one can extend the above equivalence to a weaker setting $\ric^{N}_{f_{n},M_{n}}\geq 0$ and $H_{f_{n},\bm_{n}}\geq 0$ for $N\in (-\infty,1]$ (see Remark \ref{rem:monotonicity of weighted Ricci curvature}).
Under $\ric^{N}_{f,\bm^{\perp}}\geq 0$ and $H_{f,\bm}\geq 0$ for $N\in (-\infty,1]$,
the author \cite{S1.5} has proved a dimension free inequality
\begin{equation*}
\nu_{f,1}(M)\geq \pi^{2} (2\IR M)^{-2},
\end{equation*}
of Li-Yau, Kasue type,
and a rigidity result for the equality case (see Corollary 6.5 in \cite{S1.5}, and also \cite{LY}, \cite{K2}, and cf. Remarks \ref{rem:Li-Yau dimension free} and \ref{rem:remark on the curvature assumption}).
\end{rem}

\section{Finite dimensional comparisons}\label{sec:Comparisons}

%%%%%%%%%%%%%%%%%%%%%%%%
\subsection{Estimates}
To prove Theorem \ref{thm:main theorem},
we begin with the following:
\begin{lem}\label{lem:key lemma}
Let $\bm$ be compact.
For $N\in [n,\infty)$,
let us assume $\ric^{N}_{f,\bm^{\perp}}\geq (N-1)\kappa$ and $H_{f,\bm} \geq (N-1)\lambda$.
Suppose additionally that
$\IR M\leq D$ for some $D>0$.
For $\eta \in (0,1)$,
let $\Omega \subset M$ be a Borel subset with $m_{M,f}(\Omega)\geq \eta$ and $d_{M}(\Omega,\bm)>0$.
Then we have
\begin{equation*}
s_{N,\kappa,\lambda}(d_{M}(\Omega,\bm)) \leq s_{N,\kappa,\lambda}(D)(1-\eta),
\end{equation*} 
where $s_{N,\kappa,\lambda}$ is the function defined as $(\ref{eq:one dimensional volume growth})$.
\end{lem}
\begin{proof}
We put $r:=d_{M}(\Omega,\bm)$.
Since the open $r$-neighborhood of $\bm$ and $\Omega$ are mutually disjoint,
we see
\begin{equation*}
\eta \leq m_{M,f}(\Omega)\leq 1-m_{M,f}(B_{r}(\bm)).
\end{equation*}
From $\IR M\leq D$
we derive $M=B_{D}(\bm)$;
in particular,
we have $m_{M,f}(B_{D}(\bm))=1$.
Therefore,
Theorem \ref{thm:volume comparison} leads to
\begin{equation*}
\eta \leq 1-\frac{m_{M,f}(B_{r}(\bm))}{m_{M,f}(B_{D}(\bm))}\leq 1-\frac{s_{N,\kappa,\lambda}(r)}{s_{N,\kappa,\lambda}(D)}.
\end{equation*}
This yields the lemma.
\end{proof}

One can also show the following lemma by using Theorem \ref{thm:twisted volume comparison} instead of Theorem \ref{thm:volume comparison} in the proof of Lemma \ref{lem:key lemma}.
We omit the proof.
\begin{lem}\label{lem:twisted key lemma}
Let $\bm$ be compact.
Assume $\ric^{1}_{f,\bm^{\perp}}\geq (n-1) \,\kappa \,e^{\frac{-4f}{n-1}}$ and $H_{f,\bm}\geq (n-1)\,\lambda\,e^{\frac{-2f}{n-1}}$.
Suppose additionally that
$\IR M\leq D$ and $f\leq (n-1) \delta$ for some $D>0$ and $\delta \in \mathbb{R}$.
We further assume that
one of the following holds:
\begin{enumerate}
\item $\kappa$ and $\lambda$ satisfy the convex-ball-condition;
\item $\kappa \leq 0$ and $\lambda=\sqrt{\vert \kappa \vert}$.
\end{enumerate}
For $\eta \in (0,1)$,
let $\Omega \subset M$ be a Borel subset with $m_{M,f}(\Omega)\geq \eta$ and $d_{M}(\Omega,\bm)>0$.
Then we have
\begin{equation*}
s_{n,\kappa\,e^{-4\delta},\lambda\,e^{-2\delta}}(d_{M}(\Omega,\bm)) \leq s_{n,\kappa\,e^{-4\delta},\lambda\,e^{-2\delta}}(D)(1-\eta).
\end{equation*} 
\end{lem}

Recall that
$s_{\kappa,\lambda}(t)$ is the solution of the equation $\psi''(t)+\kappa\, \psi(t)=0$ with $\psi(0)=1,\,\psi'(0)=-\lambda$.
For $N \in (1,\infty)$,
and $\kappa$ and $\lambda$ satisfying the ball-condition,
we define a function $v_{N,\kappa,\lambda}:[0,\const]\to [0,1]$ by
\begin{equation}\label{eq:ball function}
v_{N,\kappa,\lambda}(r):=\frac{\int^{\const}_{r}\,s^{N-1}_{\kappa,\lambda}(t)\,dt}{\int^{\const}_{0}\,s^{N-1}_{\kappa,\lambda}(t)\,dt}.
\end{equation}

Under the curvature condition (\ref{eq:finite dimensional curvature bound}),
we produce the following:
\begin{thm}\label{thm:weighted main theorem}
Let $\bm$ be compact.
For $N\in [n,\infty)$,
let us assume $\ric^{N}_{f,\bm^{\perp}}\geq(N-1)\kappa$ and $H_{f,\bm}\geq (N-1)\lambda$.
Then for every $\eta \in (0,1]$
the following hold:
\begin{enumerate}
\item if $\kappa$ and $\lambda$ satisfy the ball-condition,
         then
         \begin{equation*}
         \OI((M,f);-\eta) \leq v^{-1}_{N,\kappa,\lambda}(\eta);
         \end{equation*}
\item if $\kappa<0$ and $\lambda=\sqrt{\vert \kappa \vert}$,
         then
         \begin{equation*}
         \OI((M,f);-\eta) \leq \frac{1}{(N-1)\lambda}\,\log \frac{1}{\eta}.
         \end{equation*}
\end{enumerate}
\end{thm}
\begin{proof}
We may assume that $\eta<1$.
By Lemma \ref{lem:observable inscribed radius is smaller than Dirichlet separation distance},
it suffices to prove that
$\DS((M,f);\eta)$ is at most the right hand side of the desired inequality in each case.
We may assume that
$\DS((M,f);\eta)$ is positive.
We fix a Borel subset $\Omega \subset M$ with $m_{M,f}(\Omega)\geq \eta$ and $d_{M}(\Omega,\bm)>0$.

Let us consider the case where
$\kappa$ and $\lambda$ satisfy the ball-condition.
In this case,
Theorem \ref{thm:inscribed radius comparison} implies $\IR M\leq \const$.
By using Lemma \ref{lem:key lemma},
\begin{equation*}
\eta \leq 1-\frac{s_{N,\kappa,\lambda}(d_{M}(\Omega,\bm))}{s_{N,\kappa,\lambda}(\const)}=v_{N,\kappa,\lambda}(d_{M}(\Omega,\bm)).
\end{equation*}
Hence,
$d_{M}(\Omega,\bm) \leq v^{-1}_{N,\kappa,\lambda}(\eta)$.
This proves $\DS((M,f);\eta)\leq v^{-1}_{N,\kappa,\lambda}(\eta)$.
We arrive at the desired inequality.

We next consider the case where
$\kappa<0$ and $\lambda=\sqrt{\vert \kappa \vert}$.
Notice that $s_{\kappa,\lambda}(t)=e^{-\lambda t}$.
In view of Lemma \ref{lem:key lemma},
we see
\begin{equation*}
\eta \leq 1-\frac{s_{N,\kappa,\lambda}(r)}{s_{N,\kappa,\lambda}(\IR M)} \leq 1-\frac{s_{N,\kappa,\lambda}(r)}{\int^{\infty}_{0}\, s^{N-1}_{\kappa,\lambda}(t)\,dt}=e^{-(N-1)\lambda \,r},
\end{equation*}
where we put $r:=d_{M}(\Omega,\bm)$.
In particular,
$d_{M}(\Omega,\bm)$ is smaller than or equal to the right hand side of the desired one.
We obtain
\begin{equation*}
\DS((M,f);\eta)\leq \frac{1}{(N-1)\lambda}\,\log \frac{1}{\eta}.
\end{equation*}
Thus,
we complete the proof.
\end{proof}

We also have the following estimate under the condition (\ref{eq:one dimensional curvature bound}):
\begin{thm}\label{thm:twisted weighted main theorem}
Let $\bm$ be compact.
Assume
$\ric^{1}_{f,\bm^{\perp}}\geq (n-1) \,\kappa \,e^{\frac{-4f}{n-1}}$ and $H_{f,\bm}\geq (n-1)\,\lambda\,e^{\frac{-2f}{n-1}}$.
Suppose additionally that
$f\leq (n-1) \delta$ for some $\delta \in \mathbb{R}$.
Then for every $\eta \in (0,1]$
the following hold:
\begin{enumerate}
\item if $\kappa$ and $\lambda$ satisfy the convex-ball-condition,
         then
         \begin{equation*}
         \OI((M,f);-\eta) \leq v^{-1}_{n,\kappa\,e^{-4\delta},\lambda\,e^{-2\delta}}(\eta);
         \end{equation*}
\item if $\kappa<0$ and $\lambda=\sqrt{\vert \kappa \vert}$,
         then
         \begin{equation*}
         \OI((M,f);-\eta) \leq \frac{1}{(n-1)\lambda\,e^{-2\delta}}\,\log \frac{1}{\eta}.
         \end{equation*}
\end{enumerate}
\end{thm}
\begin{proof}
The proof is similar to that of Theorem \ref{thm:weighted main theorem}.
For a fixed $\Omega \subset M$ with $m_{M,f}(\Omega)\geq \eta,\,d_{M}(\Omega,\bm)>0$,
it suffices to estimate $d_{M}(\Omega,\bm)$ from above by the right hand side of the desired inequality in each case.

In the case where
$\kappa$ and $\lambda$ satisfy the convex-ball-condition,
from Theorem \ref{thm:twisted inscribed radius comparison} and Lemma \ref{lem:twisted key lemma}
we derive
\begin{equation*}
\eta \leq 1-\frac{s_{n,\kappa\,e^{-4\delta},\lambda\,e^{-2\delta}}(d_{M}(\Omega,\bm))}{s_{n,\kappa\,e^{-4\delta},\lambda\,e^{-2\delta}}(C_{\kappa\,e^{-4\delta},\lambda\,e^{-2\delta}})}=v_{n,\kappa\,e^{-4\delta},\lambda\,e^{-2\delta}}(d_{M}(\Omega,\bm)).
\end{equation*}
In particular,
$d_{M}(\Omega,\bm) \leq v^{-1}_{n,\kappa\,e^{-4\delta},\lambda\,e^{-2\delta}}(\eta)$.
This is the desired one.

If $\kappa<0$ and $\lambda=\sqrt{\vert \kappa \vert}$,
then in view of Lemma \ref{lem:twisted key lemma} we see
\begin{equation*}
\eta \leq 1-\frac{s_{n,\kappa \,e^{-4\delta},\lambda \,e^{-2\delta}}(r)}{\int^{\infty}_{0}\, s^{n-1}_{\kappa\,e^{-4\delta},\lambda\,e^{-2\delta}}(t)\,dt}=e^{-(n-1)\lambda \,e^{-2\delta} \,r},
\end{equation*}
where $r:=d_{M}(\Omega,\bm)$.
From the above inequality,
we deduce
\begin{equation*}
d_{M}(\Omega,\bm) \leq \frac{1}{(n-1)\lambda e^{-2\,\delta}}\,\log \frac{1}{\eta}.
\end{equation*}
This completes the proof of Theorem \ref{thm:twisted weighted main theorem}.
\end{proof}

%%%%%%%%%%%%%%%%%%%%%%%%
\subsection{Proof of Theorem \ref{thm:main theorem}}

In order to conclude Theorem \ref{thm:main theorem},
we calculate the invariants of the finite dimensional model spaces.
\begin{lem}\label{lem:observable and separation on model}
For every $\eta \in (0,1]$
the following hold:
\begin{enumerate}
\item for the ball-model-space $\ball$ defined as $(\ref{eq:ball model})$,
         \begin{equation*}
         \OI(\ball;-\eta)=\DS(\ball;\eta)=v^{-1}_{n,\kappa,\lambda}(\eta);
         \end{equation*}
\item for the warped-product-model-space $M^{n}_{\kappa,\lambda}$ defined as $(\ref{eq:warped product model})$,
         \begin{equation*}
         \OI(M^{n}_{\kappa,\lambda};-\eta)=\DS(M^{n}_{\kappa,\lambda};\eta)=\frac{1}{(n-1)\lambda}\,\log \frac{1}{\eta}.
         \end{equation*}
\end{enumerate}
\end{lem}
\begin{proof}
Let us present the equality for $\ball$.
We show
\begin{equation}\label{eq:separation equal ball model}
\DS(\ball;\eta)=v^{-1}_{n,\kappa,\lambda}(\eta).
\end{equation}
Let $B_{\eta} \subset \ball$ denote the closed geodesic ball with same center as $\ball$ and $m_{\ball}(B_{\eta})=\eta$.
For the radius $r_{\eta}$ of $B_{\eta}$,
we see
\begin{equation*}
\DS \left(\ball;\eta \right)=d_{\ball}\left(B_{\eta},\partial \ball \right)=\const-r_{\eta}.
\end{equation*}
It holds that
\begin{equation*}
\eta=m_{\ball}(B_{\eta})=\frac{\int^{\const}_{\const-r_{\eta}}\,s^{n-1}_{\kappa,\lambda}(t)\,dt}{\int^{\const}_{0}\,s^{n-1}_{\kappa,\lambda}(t)\,dt}=v_{n,\kappa,\lambda}\left(\const-r_{\eta}\right);
\end{equation*}
in particular,
$\const-r_{\eta}=v^{-1}_{n,\kappa,\lambda}(\eta)$.
This yields (\ref{eq:separation equal ball model}).
Since $\supp m_{\ball}$ coincides with $\ball$,
Proposition \ref{prop:Kazukawa} and (\ref{eq:separation equal ball model}) imply the desired one.

For the warped-product-model-space $M^{n}_{\kappa,\lambda}$,
the same argument as in the proof of $\ball$ leads to the desired equality.
\end{proof}

We are now in a position to prove Theorem \ref{thm:main theorem}.

\begin{proof}[Proof of Theorem \ref{thm:main theorem}]
Let $\bm$ be compact.
Assume $\ric_{\bm^{\perp}}\geq (n-1)\kappa$ and $H_{\bm}\geq (n-1)\lambda$.
Letting $f=\log \vol_{M}(M)$ and $N=n$ in Theorem \ref{thm:weighted main theorem},
for $\eta \in (0,1]$
we have the following (cf. Remark \ref{rem:finite volume}):
\begin{enumerate}
\item if $\kappa$ and $\lambda$ satisfy the ball-condition,
         then
         \begin{equation*}
         \OI(M;-\eta) \leq v^{-1}_{n,\kappa,\lambda}(\eta);
         \end{equation*}
\item if $\kappa<0$ and $\lambda=\sqrt{\vert \kappa \vert}$,
         then
         \begin{equation*}
         \OI(M;-\eta) \leq \frac{1}{(n-1)\lambda}\,\log \frac{1}{\eta}.
         \end{equation*}
\end{enumerate}
By Lemma \ref{lem:observable and separation on model},
we complete the proof of Theorem \ref{thm:main theorem}.
\end{proof}
\section{Infinite dimensional comparisons}\label{sec:Infinite dimensional comparisons}

%%%%%%%%%%%%%%%%%%%%%%%
\subsection{Relative volume comparisons}\label{sec:Relative volume comparisons}
In order to prove Theorem \ref{thm:infinite dimensional main theorem},
we develop comparison geometry of manifolds with boundary under the curvature condition (\ref{eq:infinite dimensional curvature bound}).
We notice that
in our comparison theorems in this subsection,
$m_{M,f}$ need not be a probability measure.

We first show the following Laplacian comparison:
\begin{lem}\label{lem:Laplacian comparison}
Let $z\in \bm$.
Let us assume that
$\ric^{\infty}_{f}(\gamma'_{z}(t))\geq K$ for all $t \in (0,\tau(z))$,
and $H_{f,z}\geq \Lambda$.
Then for all $t \in (0,\tau(z))$
\begin{equation*}
\Delta_{f} \rho_{\bm}(\gamma_{z}(t))\geq Kt+\Lambda.
\end{equation*}
\end{lem}
\begin{proof}
Define $h_{f,z}:=\left(\Delta_{f}\rho_{\bm}\right) \circ \gamma_{z}$.
By applying the Bochner formula (\ref{eq:Bochner formula}) to the distance function $\rho_{\bm}$,
we obtain
\begin{align*}
0 &   =   \ric^{\infty}_{f}(\gamma'_{z}(t))+\Vert \Hess \rho_{\bm}\Vert^{2}_{\HS}\left(\gamma_{z}(t)\right)-g\left(\nabla \Delta_{f}\rho_{\bm},\nabla \rho_{\bm}  \right)(\gamma_{z}(t))\\
   &\geq K-h'_{f,z}(t).
\end{align*}
It holds that $h_{f,z}(t)\to H_{f,z}$ as $t\to 0$,
and hence
\begin{equation*}
h_{f,z}(t) \geq Kt +H_{f,z}\geq Kt +\Lambda.
\end{equation*}
We arrive at the desired inequality.
\end{proof}

Furthermore,
we prove the following volume element comparison:
\begin{lem}\label{lem:volume element comparison}
Let $z\in \bm$.
Let us assume that
$\ric^{\infty}_{f}(\gamma'_{z}(t))\geq K$ for all $t \in (0,\tau(z))$,
and $H_{f,z}\geq \Lambda$.
Then for all $t_{1},t_{2} \in [0,\tau(z))$ with $t_{1}\leq t_{2}$
\begin{equation*}
\frac{\theta_{f}(t_{2},z)}{ \theta_{f}(t_{1},z)}\leq \frac{e^{-\frac{K}{2}\,t^{2}_{2}-\Lambda\,t_{2}}}{ e^{-\frac{K}{2}\,t^{2}_{1}-\Lambda\,t_{1}}    },
\end{equation*}
where $\theta_{f}(t,z)$ is defined as $(\ref{eq:volume element})$.
\end{lem}
\begin{proof}
By (\ref{eq:Laplacian representation}) and Lemma \ref{lem:Laplacian comparison},
for all $t \in (0,\tau(z))$
we see
\begin{equation*}
\frac{d}{dt}\log \frac{\theta_{f}(t,z)}{  e^{-\frac{K}{2}\,t^{2}-\Lambda\,t}  }=-\Delta_{f}\rho_{\bm}(\gamma_{z}(t))+\left(Kt +\Lambda  \right) \leq 0.
\end{equation*}
This implies the lemma.
\end{proof}

We now conclude the following relative volume comparison:
\begin{thm}\label{thm:relative volume comparison}
Let $\bm$ be compact.
Let us assume that
$\ric^{\infty}_{f,\bm^{\perp}}\geq K$ and $H_{f,\bm}\geq \Lambda$.
Then for all $r,R>0$ with $r\leq R$
\begin{equation*}
\frac{m_{M,f}(B_{R}(\bm))}{ m_{M,f}(B_{r}(\bm))}\leq \frac{  \int^{R}_{0}\,e^{-\frac{K}{2}\,t^{2}-\Lambda\,t} \,dt }{  \int^{r}_{0}\,e^{-\frac{K}{2}\,t^{2}-\Lambda\,t} \,dt }.
\end{equation*}
\end{thm}
\begin{proof}
Using Lemma \ref{lem:volume element comparison},
for all $t_{1},t_{2}\geq 0$ with $t_{1}\leq t_{2}$ we have
\begin{equation*}
\bar{\theta}_{f}(t_{2},z)\; e^{-\frac{K}{2}\,t^{2}_{1}-\Lambda\,t_{1}} \leq \bar{\theta}_{f}(t_{1},z)\; e^{-\frac{K}{2}\,t^{2}_{2}-\Lambda\,t_{2}},
\end{equation*}
where $\bar{\theta}_{f}$ is defined as (\ref{eq:extended volume element}).
Let us integrate the both sides over $[0,r]$ with respect to $t_{1}$,
and over $[r,R]$ with respect to $t_{2}$.
It follows that
\begin{equation*}
\frac{\int^{R}_{r}\bar{\theta}_{f}(t_{2},z)\,dt_{2}}{\int^{r}_{0}\bar{\theta}_{f}(t_{1},z)\,dt_{1}}\leq \frac{\int^{R}_{r}\,  e^{-\frac{K}{2}\,t^{2}_{2}-\Lambda\,t_{2}}   \,dt_{2}}{  \int^{r}_{0}\,  e^{-\frac{K}{2}\,t^{2}_{1}-\Lambda\,t_{1}}   \,dt_{1}   }.
\end{equation*}
The formula (\ref{eq:integration formula}) yields
\begin{equation*}
          \frac{m_{M,f}( B_{R}(\bm))}{ m_{M,f}( B_{r}(\bm))  }\leq 1+\frac{\int^{R}_{r}\,  e^{-\frac{K}{2}\,t^{2}_{2}-\Lambda\,t_{2}}   \,dt_{2}}{  \int^{r}_{0}\,  e^{-\frac{K}{2}\,t^{2}_{1}-\Lambda\,t_{1}}   \,dt_{1}   }
    =      \frac{  \int^{R}_{0}\,e^{-\frac{K}{2}\,t^{2}_{2}-\Lambda\,t_{2}} \,dt_{2} }{  \int^{r}_{0}\,e^{-\frac{K}{2}\,t^{2}_{1}-\Lambda\,t_{1}} \,dt_{1} }.
\end{equation*}
We complete the proof of Theorem \ref{thm:relative volume comparison}.
\end{proof}

\begin{rem}
The author \cite{S1} has shown Lemmas \ref{lem:Laplacian comparison}, \ref{lem:volume element comparison} and Theorem \ref{thm:relative volume comparison} when $K=0$ and $\Lambda=0$ (see Lemmas 3.2, 3.4 and Theorem 5.5 in \cite{S1}, and cf. Remark \ref{rem:remark on the curvature assumption}).
\end{rem}

\begin{rem}
Under the same setting as in Theorem \ref{thm:relative volume comparison},
Morgan \cite{Mo} has obtained a similar absolute volume comparison theorem of Heintze-Karcher type (see Theorem 2 in \cite{Mo}).
\end{rem}

%%%%%%%%%%%%%%%%%%%%%%%
\subsection{Distribution laws}\label{sec:Infinite dimensional model spaces and distribution laws}
Before we show Theorem \ref{thm:infinite dimensional main theorem},
we present distribution laws concerning our infinite dimensional model spaces.

We observe that
our infinite dimensional model spaces appear as the limits of a sequence of hemispheres $\{B^{n}_{\kappa/n,0}\}$,
and that of Euclidean balls $\{B^{n}_{0,\lambda/n}\}$ for $\kappa,\lambda>0$ by letting $n\to \infty$.
\begin{prop}\label{prop:ball distribution laws}
Let us assume that
$\kappa$ and $\lambda$ satisfy the ball-condition.
Then the following distribution laws hold:
\begin{enumerate}
\item if $\kappa>0$ and $\lambda=0$,
         then
         \begin{equation*}
         \frac{d m_{I,\rho_{  \partial B^{n}_{\kappa/n,0   }     }}    }{dt} \to \frac{e^{-\frac{\kappa}{2}\,t^{2}}}{\int_{I}\,e^{-\frac{\kappa}{2}\,t^{2}}\,    dt}
         \end{equation*}
         as $n\to \infty$,
         where $m_{I,\rho_{ \partial B^{n}_{\kappa/n,0   }}}$ denotes the Borel probability measure of the $\rho_{ \partial B^{n}_{\kappa/n,0   }}$-screen $I_{\rho_{ \partial B^{n}_{\kappa/n,0   }}}$ defined as $(\ref{eq:screen})$;
         in particular,
         $m_{I,\rho_{ \partial B^{n}_{\kappa/n,0   }}}$ weakly converges to the Borel probability measure of the half-Gaussian-model-space $G_{\kappa,0}$ defined as $(\ref{eq:Gaussian model space})$;
\item if $\kappa=0$ and $\lambda>0$,
         then
         \begin{equation*}
         \frac{d   m_{I,\rho_{  \partial B^{n}_{0, \lambda/n  }     }} }{dt} \to \lambda\, e^{-\lambda t};
         \end{equation*}
         as $n\to \infty$;
         in particular,
         $m_{I,\rho_{  \partial B^{n}_{0,\lambda/n   }}}$ weakly converges to the measure of the exponential-model-space $E_{\lambda}$ defined as $(\ref{eq:exponential model space})$.
\end{enumerate}
\end{prop}
\begin{proof}
We notice that
if $\kappa$ and $\lambda$ satisfy the ball-condition,
then
\begin{align}\label{eq:basic calculation}
\frac{d m_{I,\rho_{ \partial \ball    }}    }{dt}&=\frac{s^{n-1}_{\kappa,\lambda}(t)}{\int^{\const}_{0}\,s^{n-1}_{\kappa,\lambda}(t)\,dt          },\,\, s_{\kappa,\lambda}(t)=-\sqrt{\kappa+\lambda^{2}}\,s_{\kappa}(t-\const),\\ \notag
\const&=\begin{cases}
                                       \dfrac{1}{\sqrt{\kappa}}\left(\dfrac{\pi}{2}-\tan^{-1}\dfrac{\lambda}{\sqrt{\kappa}} \right) & \text{if $\kappa>0$}, \\
                                       \lambda^{-1}                         & \text{if $\kappa =0$},\\
                                       \dfrac{1}{2\,\sqrt{\vert \kappa \vert}}\,\log \dfrac{\lambda+\sqrt{\vert \kappa \vert}}{\lambda-\sqrt{\vert \kappa \vert}} & \text{if $\kappa<0$},
                                   \end{cases}
\end{align}
where $s_{\kappa}(t)$ is a unique solution of the Jacobi equation $\psi''(t)+\kappa\, \psi(t)=0$ with $\psi(0)=0,\,\psi'(0)=1$;
in particular,
if $\kappa>0$,
then
\begin{equation}\label{eq:positive explicit calculation}
s_{\kappa,\lambda}(t)=\sqrt{1+\frac{\lambda^{2}}{\kappa}}\,\cos\sqrt{\kappa} \left(t+\frac{1}{\sqrt{\kappa} }\tan^{-1}  \frac{\lambda}{\sqrt{\kappa}} \right),
\end{equation}
and if $\kappa=0$,
then $s_{0,\lambda}(t)=1-\lambda\,t$.
Therefore,
in the case where $\kappa>0$ and $\lambda=0$,
the desired convergence follows from (\ref{eq:basic calculation}), (\ref{eq:positive explicit calculation}) and
\begin{equation*}
\cos^{n-1}\frac{t}{\sqrt{n}}\to e^{-\frac{t^{2}}{2}}.
\end{equation*}
In the case where $\kappa=0$ and $\lambda>0$,
the formula (\ref{eq:basic calculation}) yields
\begin{equation*}
\frac{d m_{I,\rho_{  \partial B^{n}_{0,\lambda/n   }     }}    }{dt}=\lambda\,\left(1-\frac{\lambda}{n}\,t \right)^{n-1}\to \lambda\, e^{-\lambda t}
\end{equation*}
as $n\to \infty$.
We complete the observation.
\end{proof}

We further mention that
the exponential-model-space also appears as the limit of a sequence of warped-product-model-spaces $\{M^{n}_{\kappa/n^{2},\lambda/n,}\}$ for $\kappa<0$ and $\lambda=\sqrt{\vert \kappa \vert}$
when $n\to \infty$,
here $\lambda/n=\sqrt{\vert \kappa/n^{2} \vert}$.
\begin{prop}\label{prop:exponential distribution laws}
Let $\kappa<0$ and $\lambda=\sqrt{\vert \kappa \vert}$.
Then
\begin{equation*}
\frac{d   m_{I,\rho_{  \partial M^{n}_{\kappa/n^{2},\lambda/n}   }} }{dt} \to \lambda\, e^{-\lambda t};
\end{equation*}
as $n\to \infty$;
in particular,
$m_{I,\rho_{ \partial M^{n}_{\kappa/n^{2},\lambda/n}   }}$ weakly converges to the Borel probability measure of the exponential-model-space $E_{\lambda}$.
\end{prop}
\begin{proof}
Let us note that
if $\kappa<0$ and $\lambda=\sqrt{\vert \kappa \vert}$,
then $s_{\kappa,\lambda}(t)=e^{-\lambda t}$
and 
\begin{equation*}
\frac{d m_{I,\rho_{ \partial M^{n}_{\kappa,\lambda}    }}    }{dt}=\frac{s^{n-1}_{\kappa,\lambda}(t)}{\int_{I}\,s^{n-1}_{\kappa,\lambda}(t)\,dt          }=(n-1)\,\lambda\,e^{-(n-1)\lambda t}.
\end{equation*}
This proves the assertion.
\end{proof}

%%%%%%%%%%%%%%%%%%%%%%%
\subsection{Proof of Theorem \ref{thm:infinite dimensional main theorem}}\label{sec:Observable inscribed radius comparisons}
One can prove the following result only by replacing the role of Theorem \ref{thm:volume comparison} with that of Theorem \ref{thm:relative volume comparison} in the proof of Lemma \ref{lem:key lemma}.
The proof is left to the reader.
\begin{lem}\label{lem:infinite dimensional key lemma}
Let $\bm$ be compact.
Let us assume $\ric^{\infty}_{f,\bm^{\perp}}\geq K$ and $H_{f,\bm} \geq \Lambda$.
Suppose additionally that
$\IR M\leq D$ for some $D>0$.
For $\eta \in (0,1)$,
let $\Omega \subset M$ be a Borel subset with $m_{M,f}(\Omega)\geq \eta$ and $d_{M}(\Omega,\bm)>0$.
Then we have
\begin{equation*}
\int^{d_{M}(\Omega,\bm)}_{0}\,e^{-\frac{K}{2}\,t^{2}-\Lambda\,t} \,dt \leq (1-\eta)\,  \int^{D}_{0}\,e^{-\frac{K}{2}\,t^{2}-\Lambda\,t} \,dt.
\end{equation*} 
\end{lem}

We are now in a position to prove Theorem \ref{thm:infinite dimensional main theorem}.
\begin{proof}[Proof of Theorem \ref{thm:infinite dimensional main theorem}]
Let $\bm$ be compact.
Let us assume $\ric^{\infty}_{f,\bm^{\perp}}\geq K$ and $H_{f,\bm}\geq \Lambda$.
Suppose either (1) $K>0$ and $\Lambda \in \mathbb{R}$;
or (2) $K=0$ and $\Lambda >0$.
We note again that
in this case,
$\int^{\infty}_{0}\,e^{-\frac{K}{2}\,t^{2}-\Lambda\,t}\,dt$ is finite.

We can assume $\eta<1$.
We prove the desired statement by estimating $\DS((M,f);\eta)$ from above.
We may assume $\DS((M,f);\eta)>0$.
Let $\Omega \subset M$ denote a Borel subset with $m_{M,f}(\Omega)\geq \eta$ and $d_{M}(\Omega,\bm)>0$.
Lemma \ref{lem:infinite dimensional key lemma} leads us to
\begin{equation*}
\eta \leq 1-\frac{\int^{d_{M}(\Omega,\bm)}_{0}\,e^{-\frac{K}{2}\,t^{2}-\Lambda\,t} \,dt}{\int^{\IR M}_{0}\,e^{-\frac{K}{2}\,t^{2}-\Lambda\,t} \,dt}
       \leq \mathcal{S}_{K,\Lambda}(d_{M}(\Omega,\bm)),
\end{equation*}
where $\mathcal{S}_{K,\Lambda}:[0,\infty]\to [0,1]$ is a function defined as
\begin{equation*}
\mathcal{S}_{K,\Lambda}(r):=\frac{\int^{\infty}_{r}\,  e^{-\frac{K}{2}\,t^{2}-\Lambda\,t} \,dt}{\int^{\infty}_{0}\,e^{-\frac{K}{2}\,t^{2}-\Lambda\,t}\,dt}.
\end{equation*}
It follows that $d_{M}(\Omega,\bm) \leq \mathcal{S}^{-1}_{K,\Lambda}(\eta)$;
in particular,
$\DS((M,f);\eta)\leq \mathcal{S}^{-1}_{K,\Lambda}(\eta)$.
Here,
we can check that
\begin{align}\label{eq:calculation of infinite dimensional model}
\mathcal{S}^{-1}_{K,\Lambda}(\eta)&=\begin{cases}
                                                           \DS(G_{K,\Lambda};\eta) & \text{if $K>0$ and $\Lambda \in \mathbb{R}$}, \\
                                                           \DS(E_{\Lambda};\eta)     & \text{if $K=0$ and $\Lambda>0$},
                                                           \end{cases}\\ \notag
                                                        &=\begin{cases}
                                                           \OI(G_{K,\Lambda};-\eta) & \text{if $K>0$ and $\Lambda \in \mathbb{R}$}, \\
                                                           \OI(E_{\Lambda};-\eta)     & \text{if $K=0$ and $\Lambda>0$},
                                                           \end{cases}
\end{align}
where we used Proposition \ref{prop:Kazukawa} in the second equality.
By combining Lemma \ref{lem:observable inscribed radius is smaller than Dirichlet separation distance}, $\DS((M,f);\eta)\leq \mathcal{S}^{-1}_{K,\Lambda}(\eta)$ and (\ref{eq:calculation of infinite dimensional model}),
we obtain 
\begin{align*}
&\quad \,\,\OI((M,f);-\eta)\leq \DS((M,f);\eta)\\
                                        &\leq \begin{cases}
                                                   \OI(G_{K,\Lambda};-\eta) & \text{if $K>0$ and $\Lambda \in \mathbb{R}$}, \\
                                                   \OI(E_{\Lambda};-\eta)     & \text{if $K=0$ and $\Lambda>0$}.
                                                  \end{cases}
\end{align*}
Thus,
we conclude Theorem \ref{thm:infinite dimensional main theorem}.
\end{proof}

\section{Boundary concentration phenomena}\label{sec:Boundary concentration phenomena}

%%%%%%%%%%%%%%%%%%%%%%%
\subsection{Critical scale orders}\label{sec:Sequences of model spaces}

We first state the following assertion concerning the critical scale order of sequences of hemispheres (cf. Proposition \ref{prop:ball distribution laws}, Subsection 1.1 in \cite{GM}, Corollary 2.22 in \cite{Sh}, Theorem 8.1.1 and Corollary 8.5.7 in \cite{Sh1}):
\begin{thm}\label{thm:sequence of hemisphere}
For $\kappa>0$,
and for every $\eta \in (0,1]$
we have
\begin{equation}\label{eq:sequence of hemisphere}
\lim_{n\to \infty}\OI(B^{n}_{\kappa/n,0};-\eta)= \PI(G_{\kappa,0};1-\eta),
\end{equation}
where $G_{\kappa,0}$ is the half-Gaussian-model-space defined as $(\ref{eq:Gaussian model space})$,
and the right hand side of $(\ref{eq:sequence of hemisphere})$ is the partial inscribed radius of $G_{\kappa,0}$ defined as $(\ref{eq:partial inscribed radius})$.
In particular,
for a sequence $\{\kappa_{n}\}$ of $\kappa_{n}>0$,
the sequence $\{B^{n}_{\kappa_{n},0}\}$ is a boundary concentration family if and only if $n\,\kappa_{n}\to \infty$.
\end{thm}
\begin{proof}
Let us regard the $n$-dimensional standard sphere $M^{n}_{n/\kappa}$ with constant curvature $n/\kappa$ as an mm-space defined as (\ref{eq:normalized metric measure space with boundary}) (see Remark \ref{rem:mm-spaces}).
We observe that
$M^{n}_{n/\kappa}$ is the double of the hemisphere $B^{n}_{\kappa/n,0}$.
Based on Lemma \ref{lem:observable and separation on model} and this geometric observation,
we obtain
\begin{equation*}
\OI(B^{n}_{\kappa/n,0};-\eta)= \DS(B^{n}_{\kappa/n,0};\eta)=\frac{1}{2}\SE(M^{n}_{\kappa/n};\eta/2,\eta/2),
\end{equation*}
where $\SE(M^{n}_{\kappa/n};\eta/2,\eta/2)$ is the $(\eta/2,\eta/2)$-separation distance of $M^{n}_{\kappa/n}$ defined as (\ref{eq:Gromov separation distance}) (see Remark \ref{rem:Gromov separation distance}).
It is well-known that
the right hand side tends to $r>0$ determined by
\begin{equation*}
\frac{1-\eta}{2}=\frac{\int^{r}_{0}\,e^{-\frac{\kappa}{2}\,t^{2}}\,    dt}{\int^{\infty}_{-\infty}\,e^{-\frac{\kappa}{2}\,t^{2}}\,    dt}
\end{equation*}
as $n\to \infty$ (see e.g., Theorem 2.1 and Lemma 2.3 in \cite{Sh}).
We see that $r$ is equal to $\PI(G_{\kappa,0};1-\eta)$,
and hence (\ref{eq:sequence of hemisphere}).

For a metric measure space with boundary $X=(X,d_{X},\mu_{X})$,
\begin{equation}\label{eq:scaling formula}
\OI(cX;-\eta)=c \OI(X;-\eta)
\end{equation}
for every $c>0$,
where we set $cX:=(X,c d_{X},\mu_{X})$ (cf. Proposition 2.19 in \cite{Sh}).
Therefore,
by (\ref{eq:sequence of hemisphere}) and $B^{n}_{\kappa_{n},0}=\sqrt{\kappa\,(n\,\kappa_{n})^{-1}} \,B^{n}_{\kappa/n,0}$,
we arrive at the desired conclusion.
\end{proof}

We next investigate sequences of Euclidean balls.
\begin{thm}\label{thm:sequence of Euclidean ball}
For $\lambda>0$,
and for every $\eta \in (0,1]$
we have
\begin{equation}\label{eq:sequence of Euclidean ball}
\lim_{n\to \infty}\OI(B^{n}_{0,\lambda/n};-\eta)= \PI(E_{\lambda};1-\eta)=\frac{1}{\lambda}\log \frac{1}{\eta},
\end{equation}
where $E_{\lambda}$ is the exponential-model-space defined as $(\ref{eq:exponential model space})$.
In particular,
for a sequence $\{\lambda_{n}\}$ of $\lambda_{n}>0$,
the sequence $\{B^{n}_{0,\lambda_{n}}\}$ is a boundary concentration family if and only if $n\,\lambda_{n}\to \infty$.
\end{thm}
\begin{proof}
By $s_{0,\lambda/n}(t)=1-\lambda\,n^{-1}\,t$,
we have $v_{n,0,\lambda/n}(r)=(1-\lambda\,n^{-1}\,r)^{n}$,
where $v_{n,0,\lambda}$ is defined as (\ref{eq:ball function}).
Lemma \ref{lem:observable and separation on model} implies that
\begin{equation*}
\OI(B^{n}_{0,\lambda/n};-\eta)=\frac{n}{\lambda}\,(1-\eta^{\frac{1}{n}})\to \frac{1}{\lambda}\log \frac{1}{\eta}
\end{equation*}
as $n\to \infty$.
On the other hand,
$\PI(E_{\lambda};1-\eta)$ is equal to a positive number $r>0$ determined as $1-\eta=\int^{r}_{0}\,\lambda\,e^{-\lambda\,t}\,dt$;
in particular,
$\PI(E_{\lambda};1-\eta)=\lambda^{-1} \log \eta^{-1}$.
This proves the equalities (\ref{eq:sequence of Euclidean ball}).
Due to $B^{n}_{0,\lambda_{n}}=\lambda\,(n\,\lambda_{n})^{-1} \,B^{n}_{0,\lambda/n}$ and (\ref{eq:scaling formula}),
we complete the proof.
\end{proof}

We also provide the following result for warped-product-model-spaces:
\begin{thm}\label{thm:warped product model Levy family}
For $\kappa<0$ and $\lambda=\sqrt{\vert \kappa \vert}$,
and for every $\eta \in (0,1]$
\begin{equation*}\label{eq:warped product Levy family}
\lim_{n\to \infty}\OI(M^{n}_{\kappa/n^{2},\lambda/n};-\eta)= \PI(E_{\lambda};1-\eta)=\frac{1}{\lambda}\log \frac{1}{\eta}.
\end{equation*}
Moreover,
for a sequence $\{\kappa_{n}\}$ of $\kappa_{n}<0$,
and for a sequence $\{\lambda_{n}\}$ of $\lambda_{n}=\sqrt{\vert \kappa_{n} \vert}$,
the sequence $\{M^{n}_{\kappa_{n},\lambda_{n}}\}$ is a boundary concentration family if and only if we have $n\,\lambda_{n}\to \infty$ as $n\to \infty$.
\end{thm}
\begin{proof}
Lemma \ref{lem:observable and separation on model} yields
\begin{equation*}
\OI(M^{n}_{\kappa/n^{2},\lambda/n};-\eta)=\frac{n}{(n-1)\lambda}\,\log \frac{1}{\eta} \to \frac{1}{\lambda}\log \frac{1}{\eta}
\end{equation*}
as $n\to \infty$.
The second equality of (\ref{eq:sequence of Euclidean ball}) tells us the desired equalities.
The later assertion also immediately follows from Lemma \ref{lem:observable and separation on model}
\end{proof}

Let us give a proof of Corollary \ref{cor:ball model Levy family}.
\begin{proof}[Proof of Corollary \ref{cor:ball model Levy family}]
Let $\kappa$ and $\lambda$ satisfy the convex-ball-condition.
We will prove that
$\{\ball\}$ is a boundary concentration family.

We first consider the case of $\kappa>0$.
In this case,
we have $\lambda \geq 0$.
Due to Theorem \ref{thm:main theorem},
we have $\OI(\ball;-\eta) \leq \OI(B^{n}_{\kappa,0};-\eta)$.
Furthermore,
Theorem \ref{thm:sequence of hemisphere} tells us that
$\{B^{n}_{\kappa,0}\}$ is a boundary concentration family.
Hence,
$\{\ball\}$ is also a boundary concentration family.

When $\kappa=0$,
the desired statement follows from Theorem \ref{thm:sequence of Euclidean ball}.

In the case of $\kappa<0$,
it holds that $\lambda>\sqrt{\vert \kappa \vert}$.
In virtue of Theorem \ref{thm:main theorem} and Lemma \ref{lem:observable and separation on model},
we obtain
\begin{equation*}
\OI(\ball;-\eta) \leq \OI\bigl(M^{n}_{\kappa,\sqrt{ \vert \kappa \vert}};-\eta\bigl) = \frac{1}{(n-1)\sqrt{ \vert \kappa \vert}}\,\log \frac{1}{\eta};
\end{equation*}
in particular,
$\OI(\ball;-\eta)\to 0$ as $n\to \infty$.
Thus,
we complete the proof of Corollary \ref{cor:ball model Levy family}.
\end{proof}

%%%%%%%%%%%%%%%%%%%%%%%
\subsection{Positive dimensional cases}\label{sec:Positive dimensional cases}
In this subsection,
we summarize corollaries of Theorem \ref{thm:weighted main theorem}.
Hereafter,
let $\{(M_{n},f_{n})\}$ be a sequence of metric measure spaces with compact boundary defined as $(\ref{eq:weighted Riemannian manifold})$.
We denote by $\dim M_{n}$ the dimension of $M_{n}$.
Theorem \ref{thm:weighted main theorem} together with Theorems \ref{thm:sequence of hemisphere}, \ref{thm:sequence of Euclidean ball}, \ref{thm:warped product model Levy family} and Corollary \ref{cor:ball model Levy family} leads us to the following:
\begin{cor}\label{cor:hemisphere corollary}
Let $\{N_{n}\}$ be a sequence of integers with $N_{n} \geq \dim M_{n}$,
and let $\{\kappa_{n}\}$ be a sequence of $\kappa_{n}>0$.
Assume $\ric^{N_{n}}_{f_{n},\partial M_{n}^{\perp}}\geq (N_{n}-1)\kappa_{n}$ and $H_{f_{n},\bm_{n}}\geq 0$ for each $n$.
If $N_{n}\,\kappa_{n}\to \infty$,
then $\{(M_{n},f_{n})\}$ is a boundary concentration family.
\end{cor}

\begin{cor}
Let $\{N_{n}\}$ be a sequence of integers with $N_{n} \geq \dim M_{n}$,
and let $\{\lambda_{n}\}$ be a sequence of $\lambda_{n}>0$.
We assume $\ric^{N_{n}}_{f_{n},\partial M_{n}^{\perp}}\geq 0$ and $H_{f_{n},\bm_{n}}\geq (N_{n}-1)\lambda_{n}$ for each $n$.
If $N_{n}\,\lambda_{n}\to \infty$,
then $\{(M_{n},f_{n})\}$ is a boundary concentration family.
\end{cor}

\begin{cor}\label{cor:finite dimensional corollary}
Let $\{N_{n}\}$ be a sequence of integers with $N_{n} \geq \dim M_{n}$,
and let $\{\kappa_{n}\}$ be a sequence of $\kappa_{n}<0$.
For each $n$,
we put $\lambda_{n}:=\sqrt{\vert \kappa_{n} \vert}$.
We assume $\ric^{N_{n}}_{f_{n},\partial M_{n}^{\perp}}\geq (N_{n}-1)\kappa_{n}$ and $H_{f_{n},\bm_{n}}\geq (N_{n}-1)\lambda_{n}$.
If $N_{n}\,\lambda_{n}\to \infty$,
then $\{(M_{n},f_{n})\}$ is a boundary concentration family.
\end{cor}

\begin{cor}
Let $\{N_{n}\}$ be a sequence of integers with $N_{n} \geq \dim M_{n}$.
Assume $\ric^{N_{n}}_{f_{n},\partial M_{n}^{\perp}}\geq (N_{n}-1)\kappa$ and $H_{f_{n},\bm_{n}}\geq (N_{n}-1)\lambda$ for each $n$.
We also assume that
one of the following holds:
\begin{enumerate}
\item $\kappa$ and $\lambda$ satisfy the convex-ball-condition;
\item $\kappa<0$ and $\lambda=\sqrt{\vert \kappa \vert}$.
\end{enumerate}
If $N_{n}\to \infty$,
then $\{(M_{n},f_{n})\}$ is a boundary concentration family.
\end{cor}

Now,
we will present a concrete example concerning these corollaries, especially Corollary \ref{cor:finite dimensional corollary}:
\begin{ex}\label{ex:finite dimensional example}
Let $\{N_{n}\}$ be a sequence of integers with $N_{n} \geq n$,
and let $\{\kappa_{n}\}$ be a sequence of $\kappa_{n}<0$.
For each $n$,
put $\lambda_{n}:=\sqrt{\vert \kappa_{n} \vert}$.
Define a sequence $\{(\widetilde{M}_{n},\widetilde{f}_{n})\}$ of metric measure spaces with boundary as
\begin{align*}
\widetilde{M}_{n}&:=([0,\infty) \times \mathbb{S}^{n-1},dt^{2}+e^{-2\lambda_{n} t}\,ds_{n-1}^{2}),\\
\widetilde{f}_{n}&:=(N_{n}-n)\,\lambda_{n}\,\rho_{\partial \widetilde{M}_{n}}-\log\,\bigl( (N_{n}-1)\,\lambda_{n}\, (\vol_{\mathbb{S}^{n-1}} (\mathbb{S}^{n-1}))^{-1}\bigl).
\end{align*}
Then the sequence $\{(\widetilde{M}_{n},\widetilde{f}_{n})\}$ is a boundary concentration family if and only if we have $N_{n}\,\lambda_{n} \to \infty$.

Computing the curvatures of $\widetilde{M}_{n}$,
we see the following (cf. calculations in the proof of Lemma 3.1 in \cite{S1}):
For all $z\in \partial \widetilde{M}_{n}$ and $t>0$,
\begin{equation*}
\ric^{N_{n}}_{\widetilde{f}_{n}}(\gamma'_{z}(t))=(N_{n}-1)\kappa_{n}, \quad H_{\widetilde{f}_{n},z}=(N_{n}-1)\lambda_{n}.
\end{equation*}
Moreover,
$\theta_{\widetilde{f}_{n}}(t,z)=e^{-\widetilde{f}_{n}(z)}\,s^{N_{n}-1}_{\kappa_{n},\lambda_{n}}(t)$,
where $\theta_{\widetilde{f}_{n}}(t,z)$ is defined as (\ref{eq:volume element}).
This yields $m_{\widetilde{M}_{n},\widetilde{f}_{n}}(\widetilde{M}_{n})=1$.
Furthermore,
$\DS(  (\widetilde{M}_{n},\widetilde{f}_{n});\eta )$ is equal to $r_{n}>0$ determined by
\begin{equation*}
\vol_{\mathbb{S}^{n-1}} (\mathbb{S}^{n-1})\,\int^{\infty}_{r_{n}}\,\theta_{\widetilde{f}_{n}}(t,z)\,dt=\eta;
\end{equation*}
in particular,
from Proposition \ref{prop:Kazukawa}
we deduce
\begin{equation*}
\OI((\widetilde{M}_{n},\widetilde{f}_{n});-\eta)=\DS(  (\widetilde{M}_{n},\widetilde{f}_{n});\eta )=\frac{1}{(N_{n}-1)\lambda_{n}}\,\log\frac{1}{\eta}.
\end{equation*}
We conclude the desired statement.
\end{ex}

%%%%%%%%%%%%%%%%%%%%%%%
\subsection{One dimensional cases}\label{sec:One dimensional cases}
We next summarize corollaries of Theorem \ref{thm:twisted weighted main theorem}.
Throughout this subsection,
we always assume $\dim M_{n}=n$.
Similarly to the above subsection,
one can verify the following assertions by using Theorem \ref{thm:twisted weighted main theorem}:

\begin{cor}
Let $\{\kappa_{n}\}$ be a sequence of $\kappa_{n}>0$,
and let $\{\delta_{n}\}$ be a sequence of $\delta_{n}\in \mathbb{R}$.
Let us assume $\ric^{1}_{f_{n},\partial M_{n}^{\perp}}\geq (n-1)\kappa_{n}\,e^{\frac{-4f_{n}}{n-1}}$ and $H_{f_{n},\bm_{n}}\geq 0$ for each $n$.
Suppose additionally that
$f_{n}\leq (n-1)\delta_{n}$.
If $n\,\kappa_{n}\,e^{-4\delta_{n}}\to \infty$,
then $\{(M_{n},f_{n})\}$ is a boundary concentration family.
\end{cor}

\begin{cor}
Let $\{\lambda_{n}\}$ be a sequence of $\lambda_{n}>0$,
and let $\{\delta_{n}\}$ be a sequence of $\delta_{n}\in \mathbb{R}$.
Let us assume $\ric^{1}_{f_{n},\partial M_{n}^{\perp}}\geq 0$ and $H_{f_{n},\bm_{n}}\geq (n-1)\lambda_{n}\,e^{\frac{-2f_{n}}{n-1}}$ for each $n$.
Suppose additionally that
$f_{n}\leq (n-1)\delta_{n}$.
If $n\,\lambda_{n}\,e^{-2\delta_{n}}\to \infty$,
then $\{(M_{n},f_{n})\}$ is a boundary concentration family.
\end{cor}

\begin{cor}\label{cor:one dimensional corollary}
Let $\{\kappa_{n}\}$ be a sequence of $\kappa_{n}<0$,
and let $\{\delta_{n}\}$ be a sequence of $\delta_{n}\in \mathbb{R}$.
For each $n$,
put $\lambda_{n}:=\sqrt{\vert \kappa_{n} \vert}$.
Assume $\ric^{1}_{f_{n},\partial M_{n}^{\perp}}\geq (n-1)\kappa_{n}\,e^{\frac{-4f_{n}}{n-1}}$ and $H_{f_{n},\bm_{n}}\geq (n-1)\lambda_{n}\,e^{\frac{-2f_{n}}{n-1}}$.
Suppose additionally that
$f_{n}\leq (n-1)\delta_{n}$.
If $n\,\lambda_{n}\,e^{-2\delta_{n}}\to \infty$,
then $\{(M_{n},f_{n})\}$ is a boundary concentration family.
\end{cor}

\begin{cor}
We assume $\ric^{1}_{f_{n},\partial M_{n}^{\perp}}\geq (n-1)\kappa\,e^{\frac{-4f_{n}}{n-1}}$ and $H_{f_{n},\bm_{n}}\geq (n-1)\lambda\,e^{\frac{-2f_{n}}{n-1}}$ for each $n$.
Suppose additionally that
$f_{n}\leq (n-1)\delta$ for some $\delta \in \mathbb{R}$.
We also assume that
one of the following holds:
\begin{enumerate}
\item $\kappa$ and $\lambda$ satisfy the convex-ball-condition;
\item $\kappa<0$ and $\lambda=\sqrt{\vert \kappa \vert}$.
\end{enumerate}
Then $\{(M_{n},f_{n})\}$ is a boundary concentration family.
\end{cor}

Let us construct a concrete example for the above corollaries:
\begin{ex}\label{ex:one dimensional example}
Let $\{\kappa_{n}\}$ be a sequence of $\kappa_{n}<0$,
and let $\{\delta_{n}\}$ be a sequence of $\delta_{n}\in \mathbb{R}$.
For each $n$,
we put $\lambda_{n}:=\sqrt{\vert \kappa_{n} \vert}$.
Let us denote by $(    \mathbb{S}^{n-1}_{\kappa_{n},\lambda_{n},\delta_{n}}, ds^{2}_{n-1,\kappa_{n},\lambda_{n},\delta_{n}})$ the $(n-1)$-dimensional sphere with volume
\begin{equation*}
e^{-C_{n}}  \left(\int^{\infty}_{1}\exp\left(-\frac{(n-1)\,\lambda_{n}\,e^{-2\delta_{n}}}{2}t^{2}\right)dt \right)^{-1},
\end{equation*}
where $C_{n}:=2^{-1}\,(n-1)\,(\lambda_{n}\,e^{-2\delta_{n}}-2\,\delta_{n})$.
We define a sequence $\{(\widetilde{M}_{n},\widetilde{f}_{n})\}$ of metric measure spaces with boundary as
\begin{align*}
\widetilde{M}_{n}&:=\bigl([0,\infty) \times \mathbb{S}^{n-1}_{\kappa_{n},\lambda_{n},\delta_{n}}, \,dt^{2}+\mathcal{H}^{2}_{\kappa_{n},\lambda_{n},\delta_{n}}(t)\,ds^{2}_{n-1,\kappa_{n},\lambda_{n},\delta_{n}}\bigl),\\
\widetilde{f}_{n}&:=-\frac{n-1}{2} \log\, \bigl(\rho_{\partial \widetilde{M}_{n}}+1\bigl)+(n-1)\delta_{n},
\end{align*}
where
\begin{equation*}
\mathcal{H}_{\kappa_{n},\lambda_{n},\delta_{n}}(t):=\frac{1}{\sqrt{t+1}} \,\exp \left(-\frac{\lambda_{n}\,e^{-2\delta_{n}}}{2}\bigl((t+1)^{2}-1\bigl) \right).
\end{equation*}
If $n\,\lambda_{n}\,e^{-2\delta_{n}}\to \infty$,
then $\{(\widetilde{M}_{n},\widetilde{f}_{n})\}$ is a boundary concentration family.

We derive this statement from Corollary \ref{cor:one dimensional corollary}.
We rewrite $\widetilde{M}_{n}$ as
\begin{equation*}
\widetilde{M}_{n}=\bigl([0,\infty) \times \mathbb{S}^{n-1}_{\kappa_{n},\lambda_{n},\delta_{n}},\, dt^{2}+F^{2}_{\kappa_{n},\lambda_{n},z}(t)\,ds^{2}_{n-1,\kappa_{n},\lambda_{n},\delta_{n}}\bigl),
\end{equation*}
where for each $z\in \partial \widetilde{M}_{n}$,
\begin{align*}
s_{\widetilde{f}_{n},z}(t)&:=\int^{t}_{0}\,  e^{\frac{-2\widetilde{f}_{n}(\gamma_{z}(a))}{n-1}}\,da,\\
F_{\kappa_{n},\lambda_{n},z}(t)&:= \exp \left(  \frac{\widetilde{f}_{n}(\gamma_{z}(t))-\widetilde{f}_{n}(z)}{n-1}  \right)\,s_{\kappa_{n},\lambda_{n}}(s_{\widetilde{f}_{n},z}(t)).
\end{align*}
From this expression,
one can compute the curvatures of $\widetilde{M}_{n}$ as follows (cf. calculations in the proof of Lemmas 3.5, 5.2 and 7.1 in \cite{S2}):
For all $z\in \partial \widetilde{M}_{n}$ and $t>0$
we have
\begin{equation*}
\ric^{1}_{\widetilde{f}_{n}}(\gamma'_{z}(t))=(n-1)\kappa_{n}\,e^{\frac{-4\widetilde{f}_{n}(\gamma_{z}(t))}{n-1}},\quad H_{\widetilde{f}_{n},z}=(n-1)\lambda_{n}\,e^{\frac{-2\widetilde{f}_{n}(z)}{n-1}}.
\end{equation*}
Moreover,
$\theta_{\widetilde{f}_{n}}(t,z)=e^{-\widetilde{f}_{n}(z)}\,s^{n-1}_{\kappa_{n},\lambda_{n}}(s_{\widetilde{f}_{n},z}(t))$,
and hence
\begin{equation*}
\theta_{\widetilde{f}_{n}}(t,z)=e^{C_{n}}\,\exp \left(-\frac{(n-1)\,\lambda_{n}\,e^{-2\delta_{n}}}{2}(t+1)^{2}\right);
\end{equation*}
in particular,
we see $m_{\widetilde{M}_{n},\widetilde{f}_{n}}(\widetilde{M}_{n})=1$.
Since $\widetilde{f}_{n}\leq (n-1)\delta_{n}$,
Corollary \ref{cor:one dimensional corollary} leads us to the desired conclusion.
\end{ex}

\end{document}